%% file: main.tex
\documentclass[11pt,a4paper]{article}

\usepackage{amsmath,amsthm}
\usepackage{lmodern}
\usepackage[T1]{fontenc}
\usepackage[numbers]{natbib}
\usepackage{graphicx}
\usepackage{url}
\usepackage{amssymb}
\usepackage{mathtools}
\usepackage{bbold}
\usepackage[all]{xy}
\usepackage{xcolor}
\usepackage{mdframed}
\usepackage{makeidx} 
\usepackage{float}
\usepackage{hyphenat}
\usepackage{pdfpages}
\usepackage{array}
\usepackage{datetime}
\usepackage{enumitem}
\usepackage{hyperref}

\allowdisplaybreaks

\input{cmd}

\begin{document}

\title{On The Convergence of the Discretized Linear Static State-Based Peridynamic Equations}

\author{Lukas Pflug\thanks{FAU Competence Center Scientific Computing, Friedrich-Alexander University of Erlangen-Nuremberg, Martensstr. 5a, 91058 Erlangen, Germany}\and Michael Stingl\thanks{Department of Continuous Optimization, Friedrich-Alexander University of Erlangen-Nuremberg, Cauerstraße 11, 91058 Erlangen, Germany}\and Max Zetzmann\thanks{Department of Continuous Optimization, Friedrich-Alexander University of Erlangen-Nuremberg, Cauerstraße 11, 91058 Erlangen, Germany, Corresponding author. E-mail: max.zetzmann@fau.de}}
\date{Friday 20th February, 2026}

\maketitle

\begin{abstract}
In this paper, the convergence of the solutions for a discretized linear state-based static
peridynamic system to the corresponding continuous solution is analytically proven. To
obtain an implementable model, we further apply one-point-quadrature to the terms in
the discrete equations. The resulting system coincides with the commonly used meshfree
discretization using a regular lattice, including the possibility of using partial area algorithms to improve the numerical behavior.
We again prove convergence, this time for fixed choices of a weighting function commonly used in literature and stronger assumptions on the input data. We note however, that these assumptions are not significantly restrictive for practical purposes. In particular, 
they still allow discontinuities in the material parameters and external body forces. 
\end{abstract}

\keywords{peridynamics, state-based, meshfree,  convergence}

\input{cnt/introduction}
\input{cnt/TheLinSBMod}
\input{cnt/ConvAna}
\input{cnt/ConvNum}

\appendix

\section*{Acknowledgment}
\input{note}

\bibliography{biblio}
\bibliographystyle{plainnat}

\section{Appendix}
\input{cnt/appendix}

\end{document}

%% file: cmd.tex
\newcounter{remarkcnt}
\counterwithin{remarkcnt}{section}
\providecommand{\keywords}[1]{{\small\textbf{\textit{Keywords---}} #1}} 

\newcommand{\num}[0]{ {\scriptscriptstyle \mathrm{num} } }
\newcommand{\fnum}[0]{\operatorname{num}}
\newcommand{\dd}{\ \mathbf{d}}

\newcommand{\qedwhite}{\hfill \ensuremath{\Box}}
\renewcommand{\emptyset}[0]{\varnothing}

\newcommand{\diam}[0]{\operatorname{diam}}
\newcommand{\Tr}[0]{\operatorname{Tr}}
\newcommand{\vol}[0]{\operatorname{vol}}
\newcommand{\sgn}[0]{\operatorname{sgn}}
\newcommand{\id}[0]{\mathrm{id}}

\renewcommand{\b}[1]{\mathbf{#1}} 
\newcommand{\ca}[1]{\mathcal{#1}}

\newcommand{\supp}[0]{\text{ supp}}

\newcommand{\idL}[0]{{\mathrm{id}_{L^2}}}
\newcommand{\supscr}[1]{{\scriptscriptstyle (#1)}}

\makeatletter
\renewcommand\paragraph{\@startsection{paragraph}{4}{\z@}%
{-3.25ex\@plus -1ex \@minus -.2ex} {1.5ex \@plus .2ex} {\normalfont\normalsize\bfseries}}
\makeatother

\newcommand{\callInteg}[3]{V_{#1}V_{#2}^\supscr{#1}V_{#3}^\supscr{#1}\tau^\num(x_{#1})\rho(x_{#2}-x_{#1})\rho(x_{#3}-x_{#1})\left(x_{#2}-x_{#1}\right)\left(x_{#3}-x_{#1}\right)^\top}

\newcommand{\frclauseSB}[3]{\rho(x_{#1}-x_{#2})\rho(x_{#3}-x_{#2})\left(x_{#1}-x_{#2}\right)\left(x_{#3}-x_{#2}\right)^\top}

\newcommand{\chii}[0]{\chi_{\Omega_i}}
\newcommand{\chij}[0]{\chi_{\Omega_j}}
\newcommand{\Pk}[0]{\ca{P}_\kappa}
\newcommand{\Kk}[0]{\ca{K}_\kappa}
\newcommand{\ProjP}[0]{\mathbb{P}_\kappa}
\newcommand{\step}[0]{\mathrm{step}}
\newcommand{\LFS}[0]{\mathrm{LFS}}

\newcommand{\theoskip}[0]{\medskip}
\newcommand{\theopreskip}[0]{\medskip}

\newenvironment{theorem}[1][]{\refstepcounter{remarkcnt}\par\theopreskip\noindent\textbf{Theorem \theremarkcnt}
\if\relax\detokenize{#1}\relax
\else
(\textit{#1})
\fi}{\par\theoskip}

\newenvironment{lemma}[1][]{\refstepcounter{remarkcnt}\par\theopreskip\noindent\textbf{Lemma \theremarkcnt}
\if\relax\detokenize{#1}\relax
\else
(\textit{#1})
\fi}{\par\theoskip}

\newenvironment{corollary}[1][]{\refstepcounter{remarkcnt}\par\theopreskip\noindent\textbf{Corollary \theremarkcnt}
\if\relax\detokenize{#1}\relax
\else
(\textit{#1})
\fi}{\par\theoskip}

\newenvironment{proposition}[1][]{\refstepcounter{remarkcnt}\par\theopreskip\noindent\textbf{Proposition \theremarkcnt}
\if\relax\detokenize{#1}\relax
\else
(\textit{#1})
\fi}{\par\theoskip}

\newenvironment{assumptions}[1][]{\refstepcounter{remarkcnt}\par\theopreskip\noindent\textbf{Assumptions \theremarkcnt}
\if\relax\detokenize{#1}\relax
\else
(\textit{#1})
\fi}{\par\theoskip}

\newenvironment{remark}[1][]{\refstepcounter{remarkcnt}\par\theopreskip\noindent\textbf{Remark \theremarkcnt}
\if\relax\detokenize{#1}\relax
\else
(\textit{#1})
\fi}{\par\theoskip}

\newenvironment{definition}[1][]{\refstepcounter{remarkcnt}\par\theopreskip\noindent\textbf{Definition \theremarkcnt}
\if\relax\detokenize{#1}\relax
\else
(\textit{#1})
\fi}{\par\theoskip}

%% file: cnt/introduction.tex
\section{Introduction}
Peridynamics is a relatively new approach to modeling a solid body in continuum mechanics first introduced by Silling in \cite{SillRef}. It uses a non-local integral equation instead of partial differential equations that, in contrast to the classical theory, does not rely on any spatial derivatives of the displacement $\b{u}$, making peridynamics naturally relevant for fracture-modeling (See e.g. \cite{SillMeshFree},  \cite{WP_FRAC}, \cite{Ha2010}). It has also inspired similar models for fracture analysis (See \cite{lipton2015cohesive}, \cite{lipton2020nonlocal}, \cite{Lipton2019DynamicBF}, as well as the model introduced in \cite{JAVILI2019125} used for fracture in \cite{CPDFRAC}). Further applications include Topology Optimization with pre-imposed cracks (see \cite{Kefal2019}). 

The original peridynamic model as introduced in \cite{SillRef} is the so-called \textit{bond-based} model, which, while being simple to implement and numerically efficient, has the disadvantage of only being able to replicate a fixed Poisson's ratio. This shortcoming motivated the formulation of a generalized, so-called \textit{state-based} model. This generalized formulation allows for the modeling of materials with arbitrary bulk and shear moduli. A particularly important instance of such a state-based model is stated in \cite{lintheoofperistate},  5.3 Example 3. One should note that the bond-based model is a special case of the state-based model. Therefore, all results in this paper are also applicable without any further effort to bond-based models. If desired, a more detailed and complete overview of peridynamics can be found in \cite{SILLINGBOOK} and \cite{HandbookPDModelling}.  \par\smallskip

This paper focuses on the linearization of this state-based model, and aims to prove the convergence of the solutions for the discretized equations to the continuous solution. 
Numerical analysis on the convergence and convergence behavior in this setting has been carried out before in \cite{SELESON20162432}, \cite{ImprOnePointQuadrData} and \cite{BobarFlorConv1DPD}. In the context of the latter, we are dealing with the so-called $m$-convergence in this paper, meaning the horizon of the model will remain the same as the discretization becomes finer.

This paper aims to prove the convergence rigorously. However, proving a rate for the convergence is not part of the scope of this paper. We will do so in two different ways: A general setting allowing more general input data and in turn requiring the exact evaluation of multiple integrals in the constituents of the equation, and in a numerical setting using one-point quadrature on these integrals, but in turn requiring the specific choice of the weighting function $\omega$ given by $\zeta \mapsto \rho(\zeta)\chi_{[0,\delta)}(\|\zeta\|)$, $\rho(\zeta) = \tfrac{1}{\|\zeta\|}$. A remark generalizing the choice for the weighting function to include more regular functions $\rho$ will also be stated at the end. This choice for the weighting function is frequently found in literature, e.g. in Silling's works on the discretization of peridynamics \cite{SillMeshFree} and the PMB (prototype microelastic brittle) material defined in it. In contrast to this PMB model however, we do not yet consider fracture in this paper. 
See also \cite{oninflu}, which studies the effect of different choices of the weighting function on the non-local constitutive behavior. The discretization scheme we use for the numerical model coincides with the spatial discretization described in \cite{SillMeshFree} on a regular lattice, extended to state-based models and is frequently encountered in numerical simulations. \par\smallskip

One should note that similar rigorous convergence results for the dynamical problem given by the peridynamics-inspired non-linear model described in  \cite{lipton2015cohesive},  \cite{lipton2020nonlocal} and \cite{Lipton2019DynamicBF} have already been shown in \cite{FinElemApproxNonLoc}, \cite{liphabbcohesivenumerical} and \cite{JhaNumericalCO}. Specifically, in \cite{FinElemApproxNonLoc}, an approximation using linear continuous interpolation on triangular or tetrahedral elements was used in the context of a finite element scheme, and it was assumed that the right hand side and the initial values are in  $H^2(\Omega)^d\cap H^1_0(\Omega)^d$. Furthermore, the material parameters were assumed to be constant. 

In \cite{liphabbcohesivenumerical} and \cite{JhaNumericalCO}, a spatially piecewise constant approximation was chosen and the regularity of the input data and right-hand side was assumed to be Hölder continuous, while the material parameters remained constant. We note that the model used in \cite{liphabbcohesivenumerical}, \cite{JhaNumericalCO} and \cite{FinElemApproxNonLoc} differs from and does not include the model considered in this paper.

Other results for the linear bond-based model include \cite{DuConv}, which proved the convergence using Fourier analysis. This approach, however, imposes significant restrictions on the geometry of $\Omega$ - a finite interval was considered in one dimension, and a box in two dimensions. On those domains, convergence was proven for a generic inner approximation in the context of a Galerkin–Ritz scheme.

Most notably, none of the above papers rigorously proved the convergence of the discrete solutions when one-point quadrature was additionally applied to the discrete equations. This is relevant since most practical computations make use of quadrature to simplify the integrals involved in the underlying non-local peridynamic model. Instead, if quadrature was addressed, the convergence of the discretized solutions using one-point quadrature was only numerically compared with the bounds and rate of the solutions without one-point quadrature (See \cite{JhaNumericalCO}, Section 6).\par\smallskip
This paper aims to fill this gap between the commonly used mesh-free discretization with one-point quadrature and the continuous formulation for the state-based linear model. To the best of our knowledge, there has not yet been a rigorous proof for the convergence of the mesh-free discretization of the static linear state-based model - whether or not one-point quadrature is used - that further does not assume continuity of the material parameters or right-hand side, nor impose significant restrictions on the domain $\Omega$.\newline

The structure of this paper is as follows: after introducing the model used in this paper in Section 1, showing the convergence of the discrete solutions without one-point quadrature will be the subject of Section 2. Therein, the main result, Theorem \ref{theo:TheDiscrSolsConv}, will be proven under the Assumptions \ref{ass:GenAssump}, showing the convergence of the discrete solutions to the continuous solution in the $L^2(\Omega)^d$ norm. Notably, no continuity assumptions are needed for the material parameters or external body forces. The main step here is showing the density of the inner approximation, proven in Proposition \ref{prop:UnionOfVIsDense}, which builds on Proposition \ref{prop:ClosedZeroSetApprox}. The rest follows from Galerkin Theory.

Section 3 will introduce the numerical model in (\ref{eqn:defnumconstparts}) and (\ref{eqn:bstiffnumconst}), which results from the discrete model in Section 2 by applying one-point quadrature. To improve the practical convergence, we also include partial area weights in the numerical formulation as suggested in \cite{Seleson2018}, which can be freely chosen up to mild consistency assumptions. Then, we prove the convergence of the resulting numerical solutions in Proposition \ref{prop:ConvLHSNum} under the Assumptions \ref{ass:NumDiscr}. For this, we need stronger assumptions on both the material parameters and the right-hand side. We assume boundedness and almost everywhere local $L$-Lipschitz continuity for some $L>0$. In particular, discontinuities on sets of measure zero are still allowed. To prove Proposition \ref{prop:ConvLHSNum}, a convergence result for the linear operators associated to the corresponding discrete and continuous problems is proven in Corollary \ref{corol:NumConvCoroll}, and a uniform coercitivity result is obtained in Corollary \ref{corol:CoercResult}. These results are all based on  Lemma \ref{lemm:LFS} and Proposition \ref{prop:ClosedZeroSetApprox} leading to Proposition \ref{prop:LJfuncAreNumerical}, which allow us to deal with the discontinuities of the input data in Proposition \ref{prop:ConvLemma}, finally yielding Corollary \ref{corol:NumConvCoroll}. Finally, Remark \ref{rm:ClosingRemarkKernel} will show that the convergence proof can be easily adapted to other weighting functions $\omega$, in particular those of the form $\omega(\zeta) = \chi_{[0,\delta)}(\|\zeta\|)\rho(\zeta)$ for a $\rho \in C^0(\mathbb{R}^d)$, as well as $\omega(\zeta) =  \chi_{[0,\delta)}(\|\zeta\|) \tfrac{p_n(\|\zeta\|)}{\|\zeta\|^\alpha}$ for a suitable polynomial $p_n \geq 0$ and $\alpha \in \{0,1\}$ as described in  \cite{Seleson2018}, \cite{ImprOnePointQuadrData}, \cite{SELESON20162432}.

%% file: cnt/TheLinSBMod.tex
\subsection{The Linear State-Based Model}
Let $d\in\{2,3\}$ denote the space dimension. Following \cite{Mengesha2014}, the variational linear static state-based problem is given using the bilinear form $B$ as defined in (\ref{eqn:defB}), by 
\begin{equation}
\label{eqn:sbprobcont}
   \text{Find }u\in \ca{V} \text{  such that } \forall v\in \ca{V}: B(u,v) = \int_{\Omega} b(x) \cdot v(x) \dd x,
\end{equation}
for an open and bounded reference configuration $\Omega\subset \mathbb{R}^d$ and the external body forces $b:\Omega \rightarrow \mathbb{R}^d$. $\ca{V}$ is a subspace of $L^2(\Omega)^d$ disjoint with \begin{equation*}
    \Pi := \{f \in L^2(\Omega)^d : \exists c\in \mathbb{R}^d,\  \mathbb{R}^{d\times d} \ni A  = -A^\top: f(x) = Ax +c, \forall x\in \Omega\}.
\end{equation*}
The subspace $\ca{V}\subset L^2(\Omega)^d$ realizes the homogeneous Dirichlet conditions imposed on the problem. For example, for a given open subset $\Theta\subset\Omega$ on which the solution should be $0$, $\ca{V}$ takes the form
\begin{equation*}
    \ca{V} := \{u \in L^2(\Omega)^d : u_{|\Theta} \equiv 0\}.
\end{equation*}
Choices for $\ca{V}$ of this form are by construction disjoint with $\Pi$. This method of imposing boundary conditions has known drawbacks. Specifically, it causes unwanted artifacts at the interface between $\Theta$ and the rest of the body. Suggestions for improved approaches can be found in \cite{zhaoGenFNM}. We will stick to this method nonetheless because of its simplicity as well as its compatibility with the theory in \cite{Mengesha2014}.

To formulate the bilinear form $B$ in (\ref{eqn:sbprobcont}), we require the following non-local operators stemming from \cite{nonlocveccalclong}. 
\begin{equation*}
\begin{aligned}
    [\ca{D}_{\overline{\omega}}^*u](x) &:= -\int_{\Omega} (u(x') - u(x))\otimes\left(\frac{x'-x}{\|x'-x\|^2}\right)  \overline{\omega}(x,x') \dd x', \\
    [\ca{D}^*u](x,x') &:= -(u(x')-u(x)) \otimes \left(\frac{x'-x}{\|x'-x\|^2}\right).
\end{aligned}
\end{equation*}
These can be understood as non-local divergence and gradient operators justified by convergence results stated in \cite{nonlocveccalclong} Section 5, as well as algebraic identities that resemble those from the local operators, see  \cite{nonlocveccalclong}, Proposition 3.2. and Theorem 4.1. We further define $\Tr: \mathbb{R}^{d\times d} \rightarrow \mathbb{R}$ to be the usual trace function. 

Choosing a weighting function $\omega\in L^1(\mathbb{R}^d)$ and material parameters $\alpha, k:  \Omega \rightarrow \mathbb{R}_{>0}$, the bilinear form in (\ref{eqn:sbprobcont}) is given for $u,v\in\ca{V}$ by
\begin{equation}
\label{eqn:defB}
    \begin{aligned}
        &B(u, v)\\
        &:= \int_{\Omega} \left(k(x) - \frac{\alpha(x)m(x)}{d^2}\right) \text{Tr}([\mathcal{D}^*_{\overline{\omega}}u](x))\text{Tr}([\mathcal{D}^*_{\overline{\omega}}v](x))
        \\
        &\quad + \alpha(x) \int_{\Omega} \omega(x'-x)\|x'-x\|^2(\text{Tr}([\mathcal{D}^*u](x,x')))(\text{Tr}([\mathcal{D}^*v](x, x') ))\dd x' \dd x,
    \end{aligned}
\end{equation}
where 
\begin{equation*}
    \overline{\omega}(x,x') := \frac{d}{m(x)} \omega(x'-x)\|x'-x\|^2,\,\,\, m(x):= \int_{\Omega} \omega(x'-x)\|x'-x\|^2\dd x'.
\end{equation*}

\noindent For the well-posedness of the continuous problem and the first convergence result, we further assume the following. \newline

\begin{assumptions} (\textit{Assumptions for the Continuous Model})\label{ass:GenAssump}
\begin{enumerate}[leftmargin=2.52em]
    \item[A1] \textit{Integrability, Non-Negativity and Symmetry of the Weighting Function}: \newline 
    $\omega(\cdot) \in L^1(\mathbb{R}^d)$, with $\omega \geqq 0$ and $\omega(\zeta) = \omega(-\zeta)$ for all $\zeta \in \mathbb{R}^d$.
    \item[A2] \textit{Bounds for the Support of the Weighting Functions}: \newline
    There exist $0 < \lambda,\delta <\infty$ such that $B_{\lambda}(0) \subset \supp (\omega) \subset B_{\delta}(0)$.
    \item[A3] \textit{Boundedness of the Material Parameters}: \newline 
    $\alpha, k:\Omega \rightarrow \mathbb{R}$ measurable and bounded, i.e., there are $\alpha_0, \alpha_1, k_0, k_1 >0$ such that $\alpha_0 \leq \alpha(x) \leq \alpha_1 <\infty$ and $k_0 \leq k(x) \leq k_1 <\infty$ for all $x\in\Omega$.
    \item[A4] \textit{Regularity of the External Body Forces}:\newline
    $b\in L^2(\Omega)^d$.
    \item[A5] \textit{Regularity of the Domain}:\newline $\Omega$ is a bounded and connected Lipschitz domain. In particular, it satisfies an interior cone condition for some angle $\theta>0$ and radius $h>0$.
    \item[A6] \textit{Regularity of the Boundary Conditions}:\newline
    $\Theta\subset\Omega$ is open, non-empty and $\vol(\partial \Theta) = 0$.
\end{enumerate}
\end{assumptions}
As shown in \cite{Mengesha2014}, the assumptions above are sufficient to obtain the well-posedness of the problem stated in (\ref{eqn:sbprobcont}). Note that \cite{Mengesha2014} proves the well-posedness under even more general assumptions on $\omega$, requiring only $\omega(\cdot) \|\cdot\|^2 \in L^1(\mathbb{R}^d)$.

%% file: cnt/ConvAna.tex
\section{Convergence of an Analytical Model}
To obtain the discrete equations, a piecewise constant approximation is chosen where $\Omega$ is being partitioned as follows.

\begin{definition}[Subdivision of $\Omega$ into $(\Omega_i)_{i\in I^\kappa}$]\label{def:subdiv}
    \begin{enumerate}[leftmargin=*]
    \item Subdivide $\mathbb{R}^d$ into half-open cubes/squares $(B_i)_{i\in Z^\kappa}$ of side-length $\kappa$ and midpoints $x_i$ on a uniform grid $\kappa \mathbb{Z}^d$. The cube containing $0$ is therefore defined to be $[-\tfrac{\kappa}{2}, \tfrac{\kappa}{2}[^d$. 
    \item Define the index family corresponding to all such cubes $B_i$ that intersect $\Omega$ non-trivially ($\text{vol}(B_i\cap \Omega) >0$) as $I^\kappa\subsetneq Z^\kappa$.
    \item For all $i\in I^\kappa$, define $\Omega_i$ as the intersection of $B_i$ with $\Omega$.
\end{enumerate}
\end{definition}

\noindent With this subdivision, we define the following operators to simplify the notation.

\begin{definition}[Piecewise Constant Projection Operators]
    \label{def:ProjDiscr}\newline Let $\chi_B(x)$ denote the characteristic function for a set $B\subset \mathbb{R}^n$. Let $\Kk: (\mathbb{R}^d)^{|I^\kappa|}\rightarrow L^2(\Omega)^d$ and $\Pk:  L^2(\Omega)^d \rightarrow (\mathbb{R}^d)^{|I^\kappa|}$ be defined by
\begin{equation}
    \Kk\vec{{u}}:= \sum_{i\in I^\kappa}\vec{{u}}_i  \chi_{\Omega_i} , \,\, (\Pk{u})_i := \frac{1}{V_i} \int_{\Omega_i} {u}(x) \dd x.
\end{equation}
Further, define $\ProjP:= \Kk \circ \Pk$ as the projection from $L^2(\Omega)^d$ onto a piecewise constant approximation with respect to the $\Omega_i$.
\end{definition}

For these operators we can show the following.

\begin{lemma}[Approximation by Piecewise Constant Functions]\label{lemm:PropProjOps}Both $\Kk$ and $\Pk$ are bounded linear operators fulfilling $\Pk\circ \Kk = \id_{\mathbb{R}^{d|I^\kappa|}}$, and the $\ProjP$ are uniformly bounded in norm by $1$, that is $\|\ProjP\|\leq 1$ for all $\kappa >0$. Finally,
\begin{equation}
    \|(\idL - \ProjP){u}\|_{L^2(\Omega)^d} \xrightarrow[]{\kappa \searrow 0} 0.
\end{equation}
for all ${u}\in L^2(\Omega)^d$.
\end{lemma}
\begin{proof}The identity $\Pk \circ \Kk = \text{id}_{\mathbb{R}^{d|I^\kappa|}}$ follows directly from the definition of $\Pk$ and $\Kk$. Their boundedness easily follows from Jensen's Inequality. Similarly, the boundedness of $\ProjP$ follows from 
\begin{equation*}
    \begin{aligned}
        \|\ProjP{u}\|_{L^2(\Omega)^d}^2 &= \sum_{i} \int_{\Omega_i} \left\|\frac{1}{V_i}\int_{\Omega_i} u(x') \dd x'\right\|^2 \dd x\\
        &\stackrel{}{\leq} \sum_{i} V_i \frac{1}{V_i} \int_{\Omega_i} \|u(x)\|^2 \dd x \leq \|{u}\|_{L^2(\Omega)^d} ^2.
    \end{aligned}
\end{equation*}
Using this, we follow \cite{LinFuncAna} section 7.21 in order to prove the last convergence claim. It is sufficient to show 
\begin{equation*}
    \|(\idL - \ProjP){f}\|_{L^2(\Omega)^d} \xrightarrow[]{\kappa \searrow 0} 0 
\end{equation*}
only for $f\in C^0_0(\Omega)^d$, due to $C^0_0(\Omega)^d$ being a dense subset of $L^2(\Omega)^d$.\newline
The claim then follows from Hölder's inequality,
\begin{equation*}
    \begin{aligned}
            \|(\idL - \ProjP){f}\|_{L^2(\Omega)^d}^2 
            &\stackrel{}{\leq} \vol(\Omega) \|(\idL - \ProjP){f}\|_{L^\infty(\Omega; (\mathbb{R}^d, \|\cdot\|_2))}^2\\
            &\leq \vol(\Omega)\left(\sum_{1\leq i\leq d}\sup_{\|x'-x\|\leq\sqrt{d}\kappa} |{f}_i(x')-{f}_i(x)|^2\right) \xrightarrow[f_i \in C^0_0(\Omega)]{\kappa \searrow 0} 0.
            \end{aligned}
\end{equation*}
\end{proof} 

We now construct the familiar algebraic form of the static linear peridynamic equation as given in (\ref{eqn:finaldiscr}). The partition of $\Omega$ into $(\Omega_i)_{i\in I^\kappa}$ is also used to define
\begin{equation*}
    \ca{W}_\kappa := \left\{\sum_{i\in I^\kappa} \chi_{\Omega_i}\vec{u}_i : \vec{u}_i \in \mathbb{R}^d \text{ for all } i \in I^\kappa \right\} \subset L^\infty(\Omega)^d.
\end{equation*}
Now set $\ca{V}_\kappa := \ca{W}_\kappa \cap \ca{V}$ to impose the boundary conditions, and pose the discrete problem as  
\begin{equation}
    \label{eqn:discreteEquation}
    \begin{aligned}
        \text{Find }u_\kappa\in \ca{V}_\kappa \text{  such that } \forall v\in \ca{V}_\kappa: B(u_\kappa,v) = \int_{\Omega} b(x) \cdot v(x) \dd x.
    \end{aligned}
\end{equation}
Since $\ca{W}_\kappa$ is a finite dimensional vector space isomorphic to $({\mathbb{R}^d})^{|I^\kappa|} \cong \mathbb{R}^{|I^\kappa|d}$, one can define the discrete right hand side $\vec{b}\in (\mathbb{R}^d)^{|I^\kappa|}$ and bilinear form $\vec{B} \in (\mathbb{R}^{d\times d})^{|I^\kappa|\times |I^\kappa|}$ for all $i,j \in I^\kappa$ and $1\leq k,l\leq d$, as
\begin{equation}
\label{eqn:defSBdiscrConstit}
    \begin{aligned}
        \vec{b}_i&=(\vec{b}_i^1, \dots, \vec{b}_i^d)^\top,\,\, \vec{b}_i^k:= \int_{\Omega_i} b(x)^\top e_k \dd x \, \, \\
        \vec{B}_{ij} &= (\vec{B}_{ij}^{kl})_{kl},\,\, \vec{B}_{ij}^{kl} := B(\chi_{\Omega_i}e_k, \chi_{\Omega_j}e_l).
    \end{aligned}
\end{equation}
Further defining
\begin{equation*}
    \vec{\ca{V}}_\kappa := \left\{ \vec{b}\in (\mathbb{R}^{d})^{|I^\kappa|} : \Kk\vec{b} = \sum_{i\in I^\kappa}\chi_{\Omega_i} \vec{b}_i \in \ca{V}_\kappa \right\},
\end{equation*}
we now rewrite equation (\ref{eqn:discreteEquation}) into its algebraic form given by 
\begin{equation}
\label{eqn:eqnvector}
   \text{Find }\vec{u}_\kappa\in \vec{\ca{V}}_\kappa \text{  such that } \forall \vec{v}\in \vec{\ca{V}}_\kappa: \vec{u}_\kappa^\top\vec{B} \vec{v} = {\vec{b}}^{\,\top} \vec{v}.
\end{equation}
To prove the well-posedness of (\ref{eqn:eqnvector}), note that $\vec{\ca{V}}_\kappa$ is a linear subspace of $\mathbb{R}^{d|I^\kappa|} \cong \ca{W}_\kappa$ defined by having a zero entry at every index $i$ for which $\vol(\Theta\cap\Omega_i)>0$. Denote the set of all indices $i$ for which $\vol(\Theta\cap\Omega_i)=0$ with $J^\kappa$. Then, define $\tilde{B} \in \mathbb{R}^{d|J^\kappa|\times d|J^\kappa|}$ as the matrix resulting from $\vec{B}$ after removing every row and column with an index $i\in I^\kappa\backslash J^\kappa$. Similarly, define $\tilde{b}$ as the vector $\vec{b}$ with all entries with an index in $I^\kappa \backslash J^\kappa$ removed. Using these new definitions, (\ref{eqn:eqnvector}) reduces down to the following final form,
\begin{equation}
\label{eqn:finaldiscr}
    \text{Find } \tilde{u}_\kappa \in \mathbb{R}^{d|J^\kappa|} \text{ s.t. } \tilde{B}\tilde{u}_\kappa  = \tilde{b}.
\end{equation}
We define the following injection for convenience,
\begin{equation*}
    E: (\mathbb{R}^d)^{|J^\kappa|}\rightarrow (\mathbb{R}^d)^{|I^\kappa|},\,\, \tilde{u} \mapsto \vec{u},\, \vec{u}_i:= \begin{cases} \tilde{u}_i & i \in J^\kappa \\ 0 & \text{else.} \end{cases}
\end{equation*}
Note that $\tilde{u}$ solves (\ref{eqn:finaldiscr}) if and only if $E(\tilde{u})$ solves (\ref{eqn:eqnvector}). 
Therefore, the piecewise constant solutions $u_\kappa$ of (\ref{eqn:discreteEquation}) are in a one to one correspondence with solutions $\tilde{u}_\kappa$ of (\ref{eqn:finaldiscr}), as the Dirichlet boundary conditions ensure that any solution of (\ref{eqn:eqnvector}) must lie in the image of $E$. \newline

Due to the coercitivity of $B$ as shown in \cite{Mengesha2014}, Lemma 3, there exists a $c>0$ such that for all $u\in \ca{V}$,
\begin{equation}\label{eqn:coercBconst}
    c\|u\|_{L^2(\Omega)^d}^2 \leq B(u,u). 
\end{equation}
We define $V_i := \vol(\Omega_i) = \vol(B_i \cap \Omega)$. Using this, one obtains for all $0\not=\tilde{u} \in \mathbb{R}^{d|J^\kappa|}$, that
\begin{equation*}
\begin{aligned}
    \tilde{u}^\top \tilde{B}\tilde{u} = (E\tilde{u})^\top \vec{B} (E\tilde{u}) &= B\left(\sum_{i\in J^\kappa}\chi_{\Omega_i}\tilde{u}_i, \sum_{i\in J^\kappa}\chi_{\Omega_i}\tilde{u}_i\right) \\
    &\geq c \left\|\sum_{i\in J^\kappa}\chi_{\Omega_i}\tilde{u}_i\right\|_{L^2(\Omega)^d}^2\geq \underbrace{ c\left(\min_{i\in J^\kappa} V_i\right)\|\tilde{u}\|^2}_{>0}.
\end{aligned}
\end{equation*}
This implies that $\tilde{B}$ is symmetric positive definite and in particular invertible. This immediately yields the well-posedness of the discrete problems posed in (\ref{eqn:finaldiscr}).\newline

The proof of the convergence of the resulting sequence $(\Kk\tilde{u}_\kappa)_\kappa$ of solutions indexed over $\kappa \searrow 0$ towards the solution of the continuous problem $u$, mainly builds on a density result for the $\ca{V}_\kappa$, as the rest follows from Galerkin theory. For the sake of completeness we will still give the entire proof. 
To this end, we need the following lemma, which is a slight adaptation of the characterization of sets of measure zero. To formulate it, we define a cube to be \textit{half-open} if it is of the form $[a_1,b_1[\times....\times [a_d,b_d[$ for $a,b \in \mathbb{R}^d$.

\begin{lemma}[Characterization of Compact Sets of Measure Zero]\label{lemm:zerosetcubeslemma} Let $A$ be a compact set of measure zero. Then, for each $\epsilon >0$, there exists a finite family $(W_f^\epsilon)_{f\in F}$ of half-open cubes of uniform side length $l$, such that
\begin{equation*}
    A \subset \bigcup_{f\in F} W^\epsilon_f,\,\, \vol\left(\bigcup_{f\in F} W^\epsilon_f\right) < \epsilon.
\end{equation*} \end{lemma}

\begin{proof} See Lemma \ref{lemm:prfcube}.\end{proof}

This lemma then gives us the following important result.

\begin{proposition}[$(B_i)_{i}$ Covering of Zero Measure Sets]\label{prop:ClosedZeroSetApprox}\\
Let the $(B_i)_{i}$ be defined as in Definition \ref{def:subdiv}. Then, for any compact set $A\subset \mathbb{R}^d$ of measure zero, one has
\begin{equation*}
    0\leq \vol\left(\bigcup_{i:B_i\cap A\not=\emptyset} B_i\right)  \leq \vol\left(\bigcup_{i:\overline{B_i}\cap A\not=\emptyset} B_i\right) \xrightarrow[]{\kappa  \searrow 0} 0.
\end{equation*}
\end{proposition}

\begin{proof} Let $\epsilon>0$. Invoking Lemma \ref{lemm:zerosetcubeslemma} now gives a finite family of half-open cubes $(W_f)_{f\in F}$ of uniform side length $q$, such that 
\begin{equation}
\label{eqn:constrcubes}
    A \subset \bigcup_{f\in F} W_f,\,\, \vol\left( \bigcup_{f\in F} W_f\right) < \tfrac{1}{2} \epsilon.
\end{equation}
Now, by setting $L_f := \{i : \overline{B_i} \cap W_f \not= \emptyset\}$, one has 
\begin{equation}
    S^\kappa := \bigcup_{i:\overline{B_i}\cap A\not=\emptyset} B_i\subset \bigcup_{f\in F} \bigcup_{i \in L_f} B_i.
\end{equation}
We now claim that, uniformly for every $f \in F$, we have
\begin{equation*}
    0\leq \lim_{\kappa \searrow 0} \vol\left(\bigcup_{i \in L_f} B_i\right) - \vol(W_f) = 0,
\end{equation*}
which follows from 
\begin{equation*}
     0 \leq \vol\left(\bigcup_{i \in L_f} B_i\right) - \vol(W_f) \leq (2\kappa + q)^d  - q^d \xrightarrow[]{\kappa \searrow 0}0.
\end{equation*}
So there exists an $N \in \mathbb{N}$, such that for all $n\geq N$ and all $f\in F$, 
\begin{equation*}
    0 \leq \vol\left(\bigcup_{i \in L_f} B_i\right) - \vol(W_f)\leq (2\kappa_n + q)^d  - q^d \leq \frac{\epsilon}{2|F|},
\end{equation*}
where $L_f$ and $B_i$ depend on $n$ through $\kappa_n$. One then finally has for all $n\geq N$,
\begin{equation*}
    \vol(S^\kappa) \leq \vol\left(\bigcup_{f\in F} \bigcup_{i \in L_f} B_i\right) \leq \vol\left(\bigcup_{f\in F} W_f \right)  + \tfrac{1}{2}\epsilon \stackrel{}{<} \epsilon.
\end{equation*}
The last inequality follows from (\ref{eqn:constrcubes}).
\end{proof}
With this, we obtain the main contribution of this chapter, the density result, as follows.

\begin{proposition}[Density of the Discrete Admissible Sets]
    \label{prop:UnionOfVIsDense} For every sequence $\kappa_n\searrow 0$, the set
    \begin{equation*}
        \tilde{\ca{V}} := \bigcup_{n\in\mathbb{N}} \ca{V}_{\kappa_n}
    \end{equation*}
    is (strongly) dense in $\ca{V}$ with respect to the norm of $L^2(\Omega)^d$. 
\end{proposition}

\begin{proof} Let $v\in\ca{V}$. We already know that due to {Lemma \ref{lemm:PropProjOps}}, there is a sequence of piecewise constant functions $v'_\kappa := \ProjP(v) \in L^2(\Omega)^d$ such that $v'_\kappa \rightarrow v$ as $\kappa \searrow 0$. The index $n$ is again omitted unless relevant. One can modify this sequence to lie in $\tilde{\ca{V}}$ by restricting $v'_\kappa$ to $\cup_{i\in J^\kappa} \Omega_i =: SQ_{J^\kappa}$, thereby defining $v_\kappa := {v'_\kappa}_{|SQ_{J^\kappa}}$. Recall that $J^\kappa\subset I^\kappa$ was the set of indices $i\in I^\kappa$ for which $\vol(\Omega_i \cap \Theta) = 0$.\newline
We now split the group of the $\Omega_i$ for $i\in I^\kappa\backslash J^\kappa$ into those completely inside of $\overline{\Theta}$ and those intersecting both $\Omega\backslash \overline{\Theta}$ and $\Theta$. This leads to the following definitions,
\begin{equation*}
    \begin{aligned}
        \Omega \supset SQ_{K^\kappa}&:= \bigcup_{i\in K^\kappa} \Omega_i,\,\, K^\kappa := \{i\in I^\kappa\backslash J^\kappa : \Omega_i \not\subset \overline{\Theta}\},\\
        \Omega \supset SQ_{R^\kappa}&:= \Omega \backslash (SQ_{K^\kappa}\cup SQ_{J^\kappa}) = N \dot\cup \bigcup_{i\in R^\kappa} \Omega_i,\,\, R^\kappa := \{i\in I^\kappa: \Omega_i \subset \overline{\Theta}\},
    \end{aligned}
\end{equation*}
where $N\subset \Omega$ is some set of measure zero.
Then, one obtains
\begin{equation*}
    \begin{aligned}
        \|v_\kappa - v\|_{L^2(\Omega)^d}^2 &\leq \|\ProjP(v) - v\|_{L^2(SQ_{J^\kappa})^d}^2 + \|v_\kappa - v\|_{L^2(SQ_{K^\kappa})^d}^2 + \|v_\kappa -v\|_{L^2(SQ_{R^\kappa})^d}^2\\
        &\leq \underbrace{\|\ProjP(v) - v\|_{L^2(\Omega)^d}^2}_{\xrightarrow[]{\text{L}\ref{lemm:PropProjOps}}0} + \|v\|_{L^2(SQ_{K^\kappa})^d}^2
    \end{aligned}
\end{equation*}
since $v_\kappa \equiv 0$ on $SQ_{K^\kappa}\cup SQ_{R^\kappa}$ and $v\equiv 0$ on $SQ_{R^\kappa}$ because of $v \in \ca{V}$.

Because of the absolute continuity of the Lebesgue integral, one now only has to show that
\begin{equation*}
    \vol(SQ_{K^\kappa}) \xrightarrow[]{\kappa \searrow 0} 0.
\end{equation*}
This is where we will need assumption A6 in \ref{ass:GenAssump}, which gives us a more concrete description of $SQ_{K^\kappa}$ as 
\begin{equation*}
    SQ_{K^\kappa} \subset \bigcup_{i  : B_i \cap \partial \Theta \not= 0} B_i =: S^\kappa.
\end{equation*}
This follows from the openness of $\Theta$ and the path-connectedness of the half-open cubes $B_i$. 
More specifically, for any $\Omega_i$  with $i \in K^\kappa$, one has
\begin{equation*}
    \begin{aligned}
            \Omega_i \cap \Theta &= B_i \cap \Theta \not= \emptyset \\
            \Omega_i \cap (\Omega\backslash \overline{\Theta}) &= B_i \cap (\Omega\backslash \overline{\Theta}) \not= \emptyset.
    \end{aligned}
\end{equation*}
Now, a typical connectedness argument implies that $B_i \cap \partial\Theta \not= \emptyset$. So that $\Omega_i \subset B_i\subset S^\kappa$. The last part of the proof now uses that $\partial \Theta$ is by assumption a compact set of measure zero to apply Proposition \ref{prop:ClosedZeroSetApprox}.\end{proof}

With the density proven, the rest of the convergence follows from Galerkin theory as follows.

\begin{theorem}[The Convergence of the Discrete Solutions]\label{theo:TheDiscrSolsConv} Let the Assumptions \ref{ass:GenAssump} hold. Then, let $\kappa_n \searrow 0$. Let $u_{\kappa_n}$ be the unique solution to (\ref{eqn:discreteEquation}) for each $\kappa_n$ and $u$ the unique solution to (\ref{eqn:sbprobcont}), then
\begin{equation*}
    \|u- u_{\kappa_n}\|_{L^2(\Omega)^d} \xrightarrow[]{\kappa_n \searrow 0} 0.
\end{equation*}\end{theorem}

\begin{proof} This proof uses the fact that the bilinear form $B$ can also be represented as a self-adjoint operator $\ca{L}\in L^2(\Omega)^d \rightarrow L^2(\Omega)^d$, such that
\begin{equation*}
    B(u,v) = \langle -\ca{L} u, v \rangle_{L^2(\Omega)^d}=\langle  u, -\ca{L}v \rangle_{L^2(\Omega)^d}
\end{equation*}
for all $u,v\in L^2(\Omega)^d$ (See \cite{Mengesha2014}, Lemma 5). The proof then begins by showing a boundedness statement for the $u_\kappa$. For convenience, we drop the index $n$ from the $\kappa_n$ when it is not needed. Using the coercitivity of the bilinear form as in (\ref{eqn:coercBconst}), one has
\begin{equation*}
    0 \leq c \|u_\kappa\|^2_{L^2(\Omega)^d} \leq B(u_\kappa, u_\kappa) = \langle b,u_\kappa \rangle_{L^2(\Omega)^d} \leq \|b\|_{L^2(\Omega)^d}\|u_\kappa\|_{L^2(\Omega)^d}
\end{equation*}
leading to $\frac{1}{c} \|b\|_{L^2(\Omega)^d} \geq \|u_\kappa\|_{L^2(\Omega)^d}$ being uniformly bounded.

Applying the Banach-Alaoglu theorem, one obtains a weakly-$L^2$ convergent subsequence $u_{\kappa_{n_k}}$, renamed to just ${u}_{\kappa_n}$ for simplicity, with some weak limit $u'\in L^2(\Omega)^d$. Due to the weak closedness of $\ca{V}$, we have $u'\in \ca{V}$. 

The next step is proving $u' = u$ by showing that $u'$ solves (\ref{eqn:sbprobcont}), thereby proving that the entire original sequence weakly converges to $u$. Let $v\in \ca{V}$, then, by Proposition \ref{prop:UnionOfVIsDense}, there is a sequence of $(v_{\kappa_n})_{n\in\mathbb{N}}$, such that $v_{\kappa_n} \stackrel{}{\rightarrow} v$ strongly in $L^2(\Omega)^d$ and $v_{\kappa_n} \in \ca{V}_{\kappa_n}$ for all $n\in \mathbb{N}$. 
Then, since $\mathcal{L}$ is a bounded operator and therefore $\ca{L}v_{\kappa_n} \rightarrow \ca{L}v$ strongly in $L^2(\Omega)^d$,
\begin{equation*}
    \begin{aligned}
    B(u_{\kappa}, v_{\kappa}) &= \langle u_{\kappa}, -\mathcal{L} v_{\kappa} \rangle_{L^2(\Omega)^d} \xrightarrow[]{\kappa \searrow 0 } \langle u', -\mathcal{L}v\rangle_{L^2(\Omega)^d}.
    \end{aligned}
\end{equation*}
This immediately yields 
\begin{equation*}
    B(u',v) = \lim_{n\rightarrow \infty} B(u_{\kappa_n}, v_{\kappa_n})  \stackrel{sol}{=}  \lim_{n\rightarrow \infty}\langle b, v_{\kappa_n} \rangle_{L^2(\Omega)^d} = \langle b, v\rangle_{L^2(\Omega)^d}.
\end{equation*}
The only thing left to show now is that the series also converges to $u$ in the $L^2(\Omega)^d$ norm. This is done by using the fact that $\ca{V}_\kappa \subset \ca{V}$ and therefore $B(u, u_\kappa) = \langle u_\kappa, b\rangle_{L^2(\Omega)^d} = B(u_\kappa, u_\kappa)$ to obtain 
\begin{equation*}
    \begin{aligned}
        c\|u-u_\kappa\|^2_{L^2(\Omega)^d} &\leq B(u-u_\kappa,u-u_\kappa)= B(u,u-u_\kappa)- B(u_\kappa, u) + B(u_\kappa, u_\kappa) \\
        &\stackrel{}{=} \langle u-u_\kappa, b\rangle_{L^2(\Omega)^d} \xrightarrow[]{\kappa \searrow 0} 0.
    \end{aligned}    
\end{equation*}
We used the fact that $u$ solves (\ref{eqn:sbprobcont}) in the last equality.
\end{proof}

\begin{remark}\label{rm:BondBased} Note that by setting the material parameter $\alpha \equiv \tfrac{d^2 k}{m}$ in (\ref{eqn:defB}), the state-based model reduces to the bond-based model. Therefore Theorem \ref{theo:TheDiscrSolsConv} is applicable for bond-based models as well, as A5 of Assumptions \ref{ass:GenAssump} implies $0<m_0\leqq m \leqq m_1$ for some $m_0, m_1 > 0$ and therefore bounds on $\alpha \equiv \tfrac{d^2 k}{m}$ as required by A3. \end{remark}

%% file: cnt/ConvNum.tex
\section{Convergence of a Numerical Model}
The above framework relies on the exact computation of the $\vec{B}^{kl}_{ij}$ and $\vec{b}_i^k$ factors involved in (\ref{eqn:defSBdiscrConstit}). In practice, the exact computation of these factors is impractical, giving rise to the need to develop a numerically feasible model for which one can obtain a convergence result similar to Theorem \ref{theo:TheDiscrSolsConv} in Section 2. To this end, we now give explicit formulas for the constituents of the analytical discretization to which we want to apply one point quadrature. To make the use of one point quadrature possible, we will have to rewrite the integrals in the bilinear form into a sum of integrals on vanishing domains as $\kappa \rightarrow 0$. 

The right hand side does not need any further treatment, as it is given by
\begin{equation*}
    \vec{b}_i^k = \int_{\Omega_i} e_k^\top b(x) \dd x.
\end{equation*}
So, to deal with the $\vec{B}^{kl}_{ij}$ terms from (\ref{eqn:defSBdiscrConstit}), first writing out the terms in the definition of $B$ in (\ref{eqn:defB}) , one has
\begin{equation}
\label{eqn:explformB}
    \begin{aligned}
        B(u,v)&= \int_{\Omega} \left(k(x) - \frac{\alpha(x)m(x)}{d^2}\right) \int_{\Omega} (u(x') - u(x))^\top\left(\frac{x'-x}{\|x'-x\|^2}\right)  \overline{\omega}(x,x') \dd x' \\
        &\qquad\times\int_{\Omega} (v(x') - v(x))^\top\left(\frac{x'-x}{\|x'-x\|^2}\right)  \overline{\omega}(x,x') \dd x'\dd x
        \\
        &\quad+  \int_{\Omega} \alpha(x) \int_{\Omega} \omega(x'-x)\|x'-x\|^2(u(x')-u(x))^\top \left(\frac{x'-x}{\|x'-x\|^2}\right)\\
        &\quad\qquad\times(v(x')-v(x))^\top \left(\frac{x'-x}{\|x'-x\|^2}\right) \dd x' \dd x.
    \end{aligned}
\end{equation}
We then define 
\begin{equation*}
    \tau(x) := \frac{k(x)d^2}{m(x)^2} - \frac{\alpha(x)}{m(x)},
\end{equation*}
with which (\ref{eqn:explformB}) can further be rewritten to 
\begin{equation*}
    \begin{aligned}
        &B(u,v)\\
        &=\int_{\Omega} \tau(x) \int_{\Omega} (u(x') - u(x))^\top\left(x'-x\right) \omega(x'-x) \dd x'\\
        &\qquad \times\int_{\Omega} (v(x') - v(x))^\top\left(x'-x\right) \omega(x'-x) \dd x'\dd x
        \\
        &\quad + \int_{\Omega} \alpha(x) \int_{\Omega} \frac{\omega(x'-x)}{\|x'-x\|^2}(u(x')-u(x))^\top \left(x'-x\right)(v(x')-v(x))^\top \left(x'-x\right) \dd x' \dd x.
    \end{aligned}
\end{equation*}
Then, plugging in $u = \chi_{\Omega_i}e_k$ and $v = \chi_{\Omega_j}e_l$ yields that 
\begin{align}
        \vec{B}_{ij}^{kl}=&\int_{\Omega} \tau(x) \int_{\Omega} (\chii(x') - \chii(x))e_k^\top\left(x'-x\right) \omega(x'-x) \dd x' \nonumber \\
        &\qquad\times\int_{\Omega} (\chij(x') - \chij(x))e_l^\top\left(x'-x\right) \omega(x'-x) \dd x'\dd x\nonumber
        \\
        &\quad+ \int_{\Omega} \alpha(x) \int_{\Omega} \frac{\omega(x'-x)}{\|x'-x\|^2}(\chii(x') - \chii(x))e_k^\top \left(x'-x\right)\nonumber\\
        &\qquad\times(\chij(x') - \chij(x))e_l^\top\left(x'-x\right) \dd x' \dd x, \nonumber
\intertext{for $i\not= j$ we then obtain}\nonumber
        &=\int_{\Omega_i} \tau(x) \int_{\Omega\backslash\Omega_i} -e_k^\top\left(x'-x\right) \omega(x'-x) \dd x' \int_{\Omega_j} e_l^\top\left(x'-x\right) \omega(x'-x) \dd x'\dd x \nonumber\\
        &\quad+ \int_{\Omega_j} \tau(x) \int_{\Omega_i} e_k^\top\left(x'-x\right) \omega(x'-x) \dd x' \int_{\Omega\backslash\Omega_j}-e_l^\top\left(x'-x\right) \omega(x'-x) \dd x'\dd x \nonumber\\
        &\quad+ \int_{\Omega\backslash(\Omega_j \cup \Omega_i)} \tau(x) \int_{\Omega_i} e_k^\top\left(x'-x\right) \omega(x'-x) \dd x'\int_{\Omega_j} e_l^\top\left(x'-x\right) \omega(x'-x) \dd x'\dd x\nonumber\\
        &\quad-\int_{\Omega_i}\int_{\Omega_j}  (\alpha(x)+\alpha(x')) \frac{\omega(x'-x)}{\|x'-x\|^2}e_k^\top \left(x'-x\right)e_l^\top\left(x'-x\right) \dd x' \dd x\nonumber
\intertext{and for $i=j$ we have}\nonumber
        &= \int_{\Omega_i} \tau(x) \int_{\Omega\backslash \Omega_i} e_k^\top\left(x'-x\right) \omega(x'-x) \dd x'\int_{\Omega\backslash \Omega_i} e_l^\top\left(x'-x\right) \omega(x'-x) \dd x'\dd x\nonumber\\
        &\quad+\int_{\Omega\backslash \Omega_i} \tau(x) \int_{\Omega_i}e_k^\top\left(x'-x\right) \omega(x'-x) \dd x'\int_{\Omega_i} e_l^\top\left(x'-x\right) \omega(x'-x) \dd x'\dd x\nonumber\\
        &\quad+\int_{\Omega_i}\int_{\Omega\backslash\Omega_i} (\alpha(x)+\alpha(x')) \frac{\omega(x'-x)}{\|x'-x\|^2}e_k^\top \left(x'-x\right)e_l^\top\left(x'-x\right) \dd x' \dd x. \label{eqn:Bijklequation}
\end{align}
Summarizing this form and rewriting integrals over $\Omega\backslash\Omega_i$ or $\Omega\backslash(\Omega_i\cup \Omega_j)$ as sums over all relevant $\Omega_k$ making up the integration domain, one would then only have to evaluate the following integrals for $i,j,m\in I^\kappa, 1\leq k,l\leq d$, 
\begin{equation}
    \begin{aligned}
        \label{eqn:DefIInteg}
        \ca{I}^2_{ijm;kl} &:= \int_{\Omega_i} \tau(x) \int_{\Omega_j}e_k^\top\left(x'-x\right) \omega(x'-x) \dd x' \int_{\Omega_m} e_l^\top\left(x'-x\right) \omega(x'-x) \dd x'\dd x\\
        \ca{I}^1_{ij;kl} &:=\int_{\Omega_i}\int_{\Omega_j}(\alpha(x)+\alpha(x'))   \frac{\omega(x'-x)}{\|x'-x\|^2}e_k^\top \left(x'-x\right)e_l^\top\left(x'-x\right) \dd x' \dd x.
    \end{aligned}
\end{equation}
Using only these integrals, the bilinear form is then given by
\begin{equation}
    \begin{aligned}
    \label{eqn:sumeqI}
      \vec{B}^{kl}_{ij} = \begin{cases}
          \displaystyle\sum_{m\not\in\{i,j\}}\ca{I}^2_{mij;kl} - \sum_{m\not=i} \ca{I}^2_{imj;kl}-\sum_{m\not=j}\ca{I}^2_{jim;kl}- \ca{I}^1_{ij;kl} &\text{for } i\not=j  \\
           \displaystyle\sum_{n\not=i}\sum_{m\not=i}  \ca{I}^2_{inm;kl} + \sum_{m\not=i}\ca{I}^2_{mii;kl} + \sum_{m\not=i} \ca{I}^1_{im;kl} &\text{for } i=j.
      \end{cases}
    \end{aligned}
\end{equation}
In this chapter we will study the numerical version of the model considered above assuming $\omega(\zeta) = \rho(\zeta)\chi_{[0,\delta)}(\|\zeta\|)$ with $\rho(\zeta) = \frac{1}{\|\zeta\|}$ for all $\zeta\in \mathbb{R}^d$ and some fixed $\delta>0$. This numerical version is obtained by applying a one-point quadrature rule to the integrals in (\ref{eqn:DefIInteg}) and then using (\ref{eqn:sumeqI}) to obtain the numerical bilinear form $B^\num$.
This approach to discretize the peridynamic equations is not new, and was originally introduced in \cite{SillMeshFree} on bond-based models; it was merely brought into the usual variational framework here after being canonically extended to state-based models. One should note that \cite{SillMeshFree} only mathematically proved the convergence for one-dimensional models and twice continuously differentiable material parameters and displacements. 

For consistency between the numerical models and the continuous models, as well as the continuous PD model and the classical theory, we will substitute $\alpha(x)$ with $\frac{l(x)}{m(x)}$ for a function $l: \Omega \rightarrow \mathbb{R}$. \newline

First, we write the integrals in (\ref{eqn:DefIInteg}) in matrix form $\ca{I}^{1}_{ijm} := (\ca{I}^{1}_{ijm;kl})_{1\leq k,l\leq d}$ to simplify the notation. We then apply a one-point quadrature rule to obtain, using the notation $V_i := \vol(\Omega_i)$,
\begin{equation}
\label{eqn:defnumconstparts}
    \begin{aligned}
        &\ca{I}^{2,\num}_{ijm} := V_iV_jV_m\tau^\num(x_i)\rho(x_j-x_i)\rho(x_m-x_i)\left(x_j-x_i\right)\left(x_m-x_i\right)^\top w_{ij}w_{im}, \\
        &\ca{I}^{1,\num}_{ij} := V_iV_j \rho(x_j-x_i)(\alpha^\num(x_i)+\alpha^\num(x_j))\frac{\left(x_j-x_i\right)\left(x_j-x_i\right)^\top}{\|x_j-x_i\|\|x_j-x_i\|}w_{ij},
    \end{aligned}
\end{equation}
with, for $x\in \Omega_i$,
\begin{equation*}
    \tau^\num(x) := \frac{k(x_i)d^2 - l(x_i)}{m^\num(x_i)^2},\,\, \alpha^\num(x) = \frac{l(x_i)}{m^\num(x_i)},\,\, m^\num(x) := \sum_{j} w_{ij}V_j \rho(x_j-x_i) \|x_j-x_i\|^2.
\end{equation*}
The $w_{ij}$ terms now replace the $\chi_{[0,\delta)}(\|x_i-x_j\|)$ terms and allow the usage of partial area algorithms for better convergence behavior (See \cite{ImprOnePointQuadrData} for an in-depth analysis on partial area/volume weights and convergence behavior). There will be requirements for these terms later in Assumptions \ref{ass:NumDiscr} justifying this substitution. \par

Due to our choice to evaluate $l$, $k$ and $b$ at the mid-points $x_i$ of the $B_i$, it is necessary to extend the input data suitably beyond $\Omega$ as the $x_i$ do not necessarily lie in $\Omega$. However, any extension of the input data by a constant, for example $0$, is compatible with our regularity requirements on the input data stated later in Assumptions \ref{ass:NumDiscr}. 

\input{cnt/num_mot/nummot}

\begin{remark}\label{rm:BondBasedCompat} As outlined in Remark \ref{rm:BondBased}, to obtain the bond-based model from the state-based model on the continuous level, one would have to set $\alpha(x) := \tfrac{d^2k(x)}{m(x)}$. Doing so would require the exact evaluation of $m$ in all points $x_i$ if this model was discretized using $\alpha(x_i)$ in the above numerical integrals instead of $\alpha^\num(x_i)$. Another significant drawback is that it would not imply $\tau^\num(x_i) = 0$ if 
\begin{equation*}
    \tau^\num(x_i) = \tfrac{k(x_i)d^2}{m^\num(x_i)^2} - \tfrac{\alpha(x_i)}{m^\num(x_i)} 
\end{equation*}
had been used instead, as in general $m(x_i) \not= m^\num(x_i)$. This would mean that the discretized bond-based model would contain state-based terms. \newline
With the substitution $\alpha(x) = \tfrac{l(x)}{m(x)}$ however, the $m(x)$ is also numerically approximated. So with $l(x) = d^2 k(x)$, the discretization of the bond-based continuous model is equal to the usual discretization of the bond-based model, that is, all state-based terms vanish. This implies that the following results are all also valid for bond-based models without any further comment. Additionally, since $\alpha = \frac{c_d \mu(x)}{m(x)}$ (See \cite{PeriStatesConstMod}, Section 15) for a constant $c_d$ depending on the dimension $d$ and $\mu(x)$ being the shear modulus from the classical theory, modeling materials from classical theory in this discrete peridynamic setting does not require the exact evaluation of $m(x)$ either. 
\end{remark}

The right hand side of the continuous equation can also be numerically approximated as
\begin{equation}
\label{eqn:approxrhsEq}
    \vec{b}^{k\,\num}_{i} = b(x_i)^\top e_k 
\end{equation}
defining one possible choice for a vector $\vec{b}^\num_\kappa \in (\mathbb{R}^d)^{|I^\kappa|}$.

Any such choice leads to the final discrete problem
\begin{equation}
\label{eqn:numsyssb}
    \text{Find }\vec{u}_\kappa\in \vec{\ca{V}}_\kappa \text{  such that } \forall \vec{v}\in \vec{\ca{V}}_\kappa: \vec{u}_\kappa^\top\vec{B}^\num \vec{v} = M\vec{b}^{\num\top}_\kappa \vec{v},
\end{equation}
with the matrix $M\in (\mathbb{R}^{d\times d })^{|I^\kappa|\times |I^\kappa|}$ given by $M_{ii}=\text{diag}(V_i,...,V_i)$ and $0$ elsewhere.\newline

To solve this system, note that it can also be reduced to  
\begin{equation}
\label{eqn:rednumsb}
    \text{Find } \tilde{u}_\kappa \in \mathbb{R}^{d|J^\kappa|} \text{ s.t. } \tilde{B}^\num\tilde{u}_\kappa  = \tilde{b}^\num
\end{equation}
just like (\ref{eqn:finaldiscr}), again without changing the well-posedness or solution space compared to (\ref{eqn:numsyssb}). Note, that the dependency of $\tilde{B}^\num, \vec{B}^\num$ and $\tilde{b}^\num$ on $\kappa$ is omitted from the notation when not needed. \newline
 
The matrix $\vec{B}^\num$ is interpreted as a  $(\mathbb{R}^{d\times d})^{|I^\kappa| \times |I^\kappa|}$ matrix and given just like in (\ref{eqn:sumeqI}) by
\begin{equation}
\label{eqn:bstiffnumconst}
    \begin{aligned}
        \vec{B}^\num_{ij}= \begin{cases}
            \displaystyle\sum_{m\not\in\{i,j\}}\ca{I}^{2,\num}_{mij} - \sum_{m\not=i} \ca{I}^{2,\num}_{imj}-\sum_{m\not=j}\ca{I}^{2,\num}_{jim}- \ca{I}^{1,\num}_{ij} & \text{for }i\not= j\\
            \displaystyle\sum_{n\not=i}\sum_{m\not=i}  \ca{I}^{2,\num}_{inm} + \sum_{m\not=i}\ca{I}^{2,\num}_{mii} + \sum_{m\not=i} \ca{I}^{1,\num}_{im} & \text{for } i = j.
        \end{cases}
    \end{aligned}
\end{equation}

The goal of this chapter will be to show that the solutions $\vec{u}^\num_\kappa \in \mathbb{R}^{d|I^\kappa|}$ of (\ref{eqn:numsyssb}), embedded into $L^2(\Omega)^d$ via $\Kk$, converge to the continuous solution $u$ of (\ref{eqn:sbprobcont}), that is,
\begin{equation*}
    \Kk \vec{u}^\num_\kappa \xrightarrow[L^2(\Omega)^d]{\kappa \searrow 0} u.
\end{equation*}
For this we will need assumptions on the input data which will be specified next.

\subsection{The Convergence Proof}
For the numerical convergence proof, we assume the following. \newline

\begin{assumptions}(\textit{Assumptions on the Model for the Numerical Discretization})\label{ass:NumDiscr} \\
We assume  $\kappa < \frac{\delta}{\sqrt{d}}$ and further,
\begin{enumerate}[leftmargin=1.8em]
    \item[A1] \textit{The Choice for the Weighting Function}: We set $\forall \zeta\in\mathbb{R}^d:\omega(\zeta) = \rho(\zeta)\chi_{[0,\delta)}(\|\zeta\|)$ for some $\delta>0$ and $\rho(z) = \frac{1}{\|z\|}$.
    \item[A2] \textit{Regularity of the Material Parameters}: $l, k: {\Omega} \rightarrow \mathbb{R}$ a.e. locally $L,K$-Lipschitz continuous for some $L,K>0$ respectively, and $0<l_0 \leqq l \leqq l_1$ as well as $0<k_0\leqq k\leqq k_1$ for some constants $0<l_0,l_1, k_0,k_1 \in \mathbb{R}$.
    \item[A3] \textit{Convergence of the Right-Hand Side}: The choice of the numerical right hand side $\vec{b}_\kappa^\num \in \mathbb{R}^{d|I|^\kappa}$ has to fulfill 
    \begin{equation*}
        \|b - \Kk \vec{b}_\kappa^\num \|_{L^2(\Omega)^d} \rightarrow 0.
    \end{equation*}
    \item[A4] \textit{Limitations for the Partial Area Weights}: The $w_{ij}$ terms fulfill the following requirements for all $i,j\in I^\kappa$.
    \begin{enumerate}
        \item If $\|x_i-x_j\| \geq \delta + \sqrt{d}\frac{\kappa}{2}$, then $w_{ij} = 0$.
        \item If $\|x_i-x_j\| \leq \delta - \sqrt{d}\frac{\kappa}{2}$, then $w_{ij} = 1$.
        \item $0\leq w_{ij} \leq 1$.
        \item $w_{ij} = w_{ji}$.
    \end{enumerate}
    \item[A5] \textit{Regularity of the Domain}: $\Omega$ is a bounded and connected Lipschitz domain, in particular, it satisfies an interior cone condition for some angle $\theta>0$ and radius $h>0$.
    \item[A6] \textit{Regularity of the Dirichlet Boundary Condition}: $\Theta$ is open, non-empty with $\vol(\partial \Theta) = 0$.
\end{enumerate}
\end{assumptions}

The main idea behind the assumptions involving the a.e. local $L$-Lipschitz continuity is to allow bounded jumps (in the sense of discontinuities) on a set of measure zero, while keeping sufficient control over the function everywhere else, uniform with respect to $x$. Besides the material parameters, this can also be applied to the external body forces $b$. This is shown in Corollary \ref{corol:BRHS}, which will prove that choosing $b$ to be a.e. locally $L$-Lipschitz is a sufficient condition for $\vec{b}^\num_\kappa := (b(x_i))_{i\in I^\kappa}$ to fulfill the assumptions above.\newline

We begin with the following fundamental lemma that allows extending the local Lipschitz continuity to a global one on convex sets that don't intersect the set where the function behaves undesirably. 

\begin{lemma}[Globalization of a.e. Local $L$-Lipschitz Continuity]\label{lemm:LFS} \newline Let $f: {\Omega} \rightarrow \mathbb{R}^n$ be a bounded, a.e. locally $F$-Lipschitz continuous function with $F>0$, and 
\begin{equation*}
    \LFS^F(f) := \{ x \in {\Omega} : f \text{ fails to be locally }F\text{-Lipschitz in } x\}.
\end{equation*}
Then, for all convex sets $B \subset {\Omega}$ with $B \cap \LFS^F(f) = \emptyset$, we have $f_{|B} \in C^{0,1}(B)^n$ with the same Lipschitz constant $F$.\end{lemma} 
\begin{proof}
    By assumption, for any $x,x'\in B$, the line segment between $x$ and $x'$ lies in $B$. Since this line segment is also compact, it can be covered by finitely many neighborhoods on which $f$ is $F$-Lipschitz. Appropriately subdividing the line segment between $x$ and $x'$ and applying the $F$-Lipschitz bound on each segment then yields the claim after summing the bounds back together. 
\end{proof}

The next proposition states a convergence result for these functions.

\begin{proposition}[Convergence of Midpoint Approximation]\label{prop:LJfuncAreNumerical} Let $u:{\Omega}\rightarrow \mathbb{R}^d$ be a bounded, almost everywhere locally $L$-Lipschitz function, then
\begin{equation*}
    \|\fnum_\kappa(u) - u\|_{L^2(\Omega)^d} \xrightarrow[]{\kappa \searrow 0} 0,
\end{equation*}
with 
\begin{equation*}
    L^2(\Omega)^d \ni \fnum_\kappa(u) := \sum_{i\in I^\kappa} \chi_{\Omega_i}  u(x_i). 
\end{equation*}\end{proposition}
\begin{proof} We first split $I^\kappa$ into two sets as follows,
\begin{equation*}
    \hat{I} := \{i\in I^\kappa : \overline{B_i} \cap (\partial \Omega \cup \LFS^L({u})) \not= \emptyset \},\,\,\, \hat{II} := I^\kappa \backslash \hat{I}. 
\end{equation*}
One then has, using $\|\cdot\|_\infty$ to denote the supremum norm,
\begin{equation*}
    \begin{aligned}
        \|\fnum_\kappa(u) -u\|^2_{L^2(\Omega)^d} &= \sum_{i\in I^\kappa} \int_{\Omega_i}\|u(x)-u(x_i)\|^2 \dd x\\
        &=\sum_{i\in \hat{I}} \int_{\Omega_i}\|u(x)-u(x_i)\|^2 \dd x +\sum_{i\in \hat{II}} \int_{\Omega_i}\|u(x)-u(x_i)\|^2 \dd x \\
        &\stackrel{}{\leq} 4\|u\|_\infty^2 \vol\left(\bigcup_{i \in \hat{I}} B_i\right)+\sum_{i\in \hat{II}} \int_{\Omega_i}\|u(x)-u(x_i)\|^2 \dd x. \\
    \end{aligned}
\end{equation*}
Invoking Proposition \ref{prop:ClosedZeroSetApprox} now yields
\begin{equation*}
    \vol\left(\bigcup_{i \in \hat{I}} B_i \right) = \vol\left(\bigcup_{i : \overline{B_i} \cap S \not= \emptyset} B_i \right)\xrightarrow{\kappa \searrow 0} 0,
\end{equation*}
since $S:= (\partial \Omega \cup \LFS^L({u}))$ is by definition a compact zero set.

Further, $i \in \hat{II}$ implies that $\Omega_i = B_i$ as $B_i$ is path connected and $\Omega$ is open, so that one could apply a connectedness argument to show
\begin{equation*}
    \Omega_i \not= B_i \Longrightarrow \emptyset \not= \overline{B_i} \cap \partial \Omega \subset \overline{B_i} \cap (\partial \Omega \cup \LFS^L({u})) \Longrightarrow i \not\in\hat{II}.
\end{equation*}
Due to Lemma \ref{lemm:LFS}, one obtains that $u$ has to be L-Lipschitz continuous on $\Omega_i = B_i$ due to $\overline{B_i}\cap \LFS^L({u}) = \emptyset$, ultimately yielding
\begin{equation*}
    \sum_{i\in \hat{II}} \int_{\Omega_i}\|u(x)-u(x_i)\|^2 \dd x \leq L^2\vol(\Omega)d \kappa^2 \xrightarrow[]{\kappa \searrow 0} 0,
\end{equation*}
finishing the proof.
\end{proof}

\begin{corollary}[Sufficient Conditions for Assumption A3] \label{corol:BRHS} Proposition \ref{prop:LJfuncAreNumerical} implies that the choice $\vec{b}^\num_\kappa =: (b(x_i))_{i\in I^\kappa}$ as defined in (\ref{eqn:approxrhsEq}) fulfills A3 in Assumptions \ref{ass:NumDiscr} if $b$ is bounded and a.e. locally $L$-Lipschitz for some $L>0$. \qedwhite\end{corollary}

\begin{remark} \label{rm:WeightsAreWellBehaved} Note that both choices for weights used in Remark \ref{rm:WhyWeights} fulfill requirement A4 in Assumptions \ref{ass:NumDiscr}. In fact, all that A4 requires through the constraints (a) and (b) is that elements that geometrically fully lie outside or inside the neighborhood of $x_i$ due to the distance of $x_j$ to $x_i$ have to have exact weights, which we used in Lemma \ref{lemm:BoundOnDiffOfKappa}. These conditions are fulfilled by many practically used weights since testing the distance $\|x_j-x_i\|$ is already part of the algorithm to calculate the weights, or is otherwise feasible. For example, additionally to the weights in Remark \ref{rm:WhyWeights} and those suggested in \cite{Scabbia2023}, the PD-LAMMPS algorithm outlined in \cite{ImprOnePointQuadrData} also fulfills these assumptions. However, the assumptions as they are posed here are too strict for the weights mentioned in \cite{LIU2024109115} based on Monte Carlo integration to approximate the exact weights in an efficient manner. They are also too strict for the weights constructed in \cite{TRASK2019151}, aiming to find an asymptotically consistent quadrature rule resulting from a least squares problem under the constraint that a given subspace is being exactly reproduced. This is because last two choices do not guarantee that the weights $w_{ij}$ fulfill requirement (b) of A4.
\end{remark}

The next steps to prove the convergence of this numerical model are structured as follows. We will first rewrite the numerical bilinear form into two equivalent forms to show point-wise convergence of the induced linear operator as well as coercitivity. These reformulations are lengthy, yet in essence no different than in the continuous case (See \cite{Mengesha2014}, proof of Lemma 5, Proposition 1).
We begin with the operator form $\ca{L}^\num_\kappa$, motivated by the continuous equivalent
\begin{equation*}
    B(u,v) = \langle u, -\ca{L} v\rangle_{L^2(\Omega)^d}  \,\,\longrightarrow\,\, \vec{u}\vec{B}^\num \vec{v} =: \langle \Kk \vec{u}, -\ca{L}^\num_\kappa \Kk \vec{v} \rangle_{L^2(\Omega)^d},
\end{equation*}
then we will use the potential form $W^\num_\kappa$, motivated by
\begin{equation*}
    B(u,u) = 2 \int_{\Omega} W(u,x)\dd x \,\, \longrightarrow \,\, \vec{u}\vec{B}^\num\vec{u}=: 2\sum_{i}V_iW^\num_\kappa(\vec{u}, x_i).
\end{equation*}
The first step is then to show that, for some $C>0$, 
\begin{equation}
\label{eqn:numconv}
     \forall v\in L^2(\Omega)^d: \ca{L}^\num_\kappa v \xrightarrow[L^2(\Omega)^d]{\kappa \searrow 0} \ca{L}v,\,\,\text{ and } \|\ca{L}^\num_\kappa\| \leqq C \text{ independent of } \kappa,
\end{equation}
which will be done in Corollary \ref{corol:NumConvCoroll}.
The second step will be to prove that for $\kappa>0$ small enough and for all $\vec{u} \in \vec{\ca{V}}_\kappa$,
\begin{equation}
\label{eqn:numcoerc}
    c^\num\|\Kk \vec{u}\|_{L^2(\Omega)^d}^2\leq 2\sum_{i}V_iW^\num_\kappa(\vec{u}, x_i)
\end{equation}
for a $c^\num >0$ independent of $0<\kappa<<1$. This will be proven later in Corollary \ref{corol:CoercResult}.\newline
The latter property is based on the proof of the coercitivity of the continuous bilinear form (See \cite{Mengesha2014}, Proposition 1 and Proposition 2).  
These two requirements are sufficient for us to now prove the convergence. 

\begin{theorem}[Convergence of the Numerical Discretization]\label{prop:ConvLHSNum} Let $\vec{u}^\num_\kappa$ be the solutions to the numerical discrete systems (\ref{eqn:numsyssb}), then under Assumptions \ref{ass:NumDiscr}, as well as (\ref{eqn:numconv}) and (\ref{eqn:numcoerc}), 
\begin{equation*}
    \Kk \vec{u}^\num_\kappa \xrightarrow[L^2(\Omega)^d]{\kappa \searrow 0} u.
\end{equation*}\end{theorem}

\begin{proof} Let $\kappa>0$ be small enough for  (\ref{eqn:numcoerc}) to hold, as well as
\begin{equation*}
    \|\Kk \vec{b}^\num_\kappa - b\|_{L^2(\Omega)^d} \leq \tfrac{1}{2} \|b\|
\end{equation*}
made possible by $\Kk\vec{b}^\num_\kappa  \rightarrow b$ as assumed in A3 of Assumptions \ref{ass:NumDiscr}. Then, using (\ref{eqn:numcoerc}) and the fact that $\vec{u}_\kappa^\num$ is the solution of (\ref{eqn:numsyssb}),
\begin{equation*}
      \begin{aligned}
           c^\num\|\Kk  \vec{u}^\num_\kappa\|_{L^2(\Omega)^d}^2&\stackrel{}{\leq} 2\sum_{i}V_iW^\num_\kappa( \vec{u}^\num_\kappa, x_i) = \vec{u}^{\num\top}_\kappa \vec{B}^\num \vec{u}^\num_\kappa \\
           &\stackrel{}{=} M\vec{b}^{\num\top}_\kappa  \vec{u}^\num_\kappa = \langle \Kk \vec{b}^{\num}_\kappa, \Kk  \vec{u}^\num_\kappa \rangle_{L^2(\Omega)^d} \\
           &\leq \|\Kk \vec{b}^{\num}_\kappa\|_{L^2(\Omega)^d}\| \Kk \vec{u}^\num_\kappa\|_{L^2(\Omega)^d} \\
           &\leq 1.5\|b\|_{L^2(\Omega)^d}\|\Kk \vec{u}^\num_\kappa\|_{L^2(\Omega)^d},
      \end{aligned}
\end{equation*}
implying that the sequence $\Kk \vec{u}^\num_\kappa$ is uniformly bounded in $L^2(\Omega)^d$. By applying Banach-Alaoglu, one obtains a subsequence, written like the original sequence for convenience, and a weak limit $u'\in L^2(\Omega)^d$. This weak limit must lie in $\ca{V}$ due to the weak closedness of $\ca{V}$. \newline
We will now show $u' = u$. Let $v\in \ca{V}$. Using Proposition \ref{prop:UnionOfVIsDense}, there exist $\vec{v}_\kappa \in \vec{\ca{V}}_\kappa$ such that $\Kk \vec{v}_\kappa \rightarrow v$ in $L^2(\Omega)^d$, then,
\begin{equation*}
    M\vec{b}^{\num\top}_\kappa \vec{v}_\kappa = \langle \Kk \vec{b}^{\num}_\kappa, \Kk \vec{v}_\kappa \rangle_{L^2(\Omega)^d} \xrightarrow[]{\kappa \searrow 0} \langle b, v\rangle_{L^2(\Omega)^d},
\end{equation*}
as well as, using (\ref{eqn:numconv}),
\begin{equation*}
    M\vec{b}^{\num\top}_\kappa \vec{v}_\kappa 
 \stackrel{}{=} \vec{u}_\kappa^{\num\top}\vec{B}^\num\vec{v}_\kappa \stackrel{}{=} \langle \Kk \vec{u}_\kappa^{\num}, -\ca{L}^\num_\kappa\Kk \vec{v}_\kappa\rangle_{L^2(\Omega)^d}.
\end{equation*}
Further, again with (\ref{eqn:numconv}),
\begin{equation*}
    \|\ca{L}^\num_\kappa\Kk \vec{v}_\kappa - \ca{L}v\|_{L^2(\Omega)^d} \leq \|(\ca{L}^\num_\kappa- \ca{L})v\|_{L^2(\Omega)^d} + \|\ca{L}^\num_\kappa\|\|v-\Kk\vec{v}_\kappa\|_{L^2(\Omega)^d} \xrightarrow[]{\kappa \searrow 0} 0
\end{equation*}
leads to 
\begin{equation*}
    \langle \Kk \vec{u}^{\num\top}_\kappa, -\ca{L}^\num_\kappa\Kk \vec{v}_\kappa\rangle_{L^2(\Omega)^d} \xrightarrow[ ]{\kappa \searrow 0} \langle u', -\ca{L} v\rangle_{L^2(\Omega)^d} .
\end{equation*}
In summary,
\begin{equation*}
    \langle u', -\ca{L} v \rangle_{L^2(\Omega)^d}  = \langle b, v\rangle_{L^2(\Omega)^d} 
\end{equation*}
for all $v\in \ca{V}$, therefore $u'=u$. 
Further, the uniqueness of $u$ implies that the whole sequence $\Kk \vec{u}^\num_\kappa$ is weakly convergent to $u$.

We now show that the sequence $\Kk \vec{u}^\num_\kappa$ also strongly converges to $u$ with respect to the $L^2(\Omega)^d$ norm. Let $\vec{s}_\kappa \in \vec{\ca{V}}_\kappa$ such that $\Kk \vec{s}_\kappa \rightarrow u$ in $L^2(\Omega)^d$ obtained just like in Proposition \ref{prop:UnionOfVIsDense}.
We then have the inequality
\begin{equation*}
    \|\Kk \vec{u}_\kappa^\num -u\|_{L^2(\Omega)^d} \leq \|\Kk (\vec{u}_\kappa^\num - \vec{s}_\kappa)\|_{L^2(\Omega)^d} + \|\Kk \vec{s}_\kappa -u\|_{L^2(\Omega)^d},
\end{equation*}
where we know the second summand vanishes. Turning our attention to the first, for $\kappa>0$ small enough, one has using (\ref{eqn:numcoerc}),
\begin{equation*}
    \begin{aligned}
         c^\num\|\Kk (\vec{u}_\kappa^\num - \vec{s}_\kappa)\|^2_{L^2(\Omega)^d} &\stackrel{}{\leq} (\vec{u}_\kappa^\num - \vec{s}_\kappa)^\top \vec{B}^\num(\vec{u}_\kappa^\num - \vec{s}_\kappa) \\
        &\leq (\vec{u}_\kappa^\num)^\top \vec{B}^\num(\vec{u}_\kappa^\num -\vec{s}_\kappa)+ \vec{s}_\kappa^\top \vec{B}^\num \vec{s}_\kappa - \vec{s}_\kappa^\top \vec{B}^\num\vec{u}_\kappa^\num \\
        &\stackrel{}{=} M\vec{b}^{\num\top}_\kappa(\vec{u}_\kappa^\num -\vec{s}_\kappa) + \vec{s}_\kappa^\top \vec{B}^\num \vec{s}_\kappa - M\vec{b}^{\num\top}_\kappa\vec{s}_\kappa.
    \end{aligned}
\end{equation*}
Where in the last line, we used that $\vec{u}_\kappa^\num$ solves (\ref{eqn:numsyssb}). 

Note that $\Kk \vec{s}_\kappa \rightarrow u$, $\Kk \vec{u}_\kappa^\num \rightharpoonup u$ and $\Kk\vec{b}^\num_\kappa \rightarrow b$, which leads to
\begin{equation*}
    M\vec{b}^{\num\top}_\kappa(\vec{u}_\kappa^\num -\vec{s}_\kappa) =\langle \Kk\vec{b}^{\num}_\kappa, \Kk (\vec{u}_\kappa^\num -\vec{s}_\kappa )\rangle_{L^2(\Omega)^d} \xrightarrow[]{\kappa \searrow 0} 0,
\end{equation*}
as well as
\begin{equation*}
    M\vec{b}^{\num\top}_\kappa\vec{s}_\kappa =\langle \Kk \vec{b}^{\num}_\kappa, \vec{s}_\kappa \rangle_{L^2(\Omega)^d} \xrightarrow[]{\kappa \searrow 0} \langle b, u\rangle_{L^2(\Omega)^d}.
\end{equation*}
So what's left to show is
\begin{equation*}
    \vec{s}_\kappa^\top \vec{B}^\num\vec{s}_\kappa \rightarrow \langle b,u \rangle_{L^2(\Omega)^d} = B(u,u).
\end{equation*}
First, we see that 
\begin{equation*}
    \begin{aligned}
        \vec{s}_\kappa^\top \vec{B}^\num\vec{s}_\kappa &= \langle \Kk \vec{s}_\kappa, -\ca{L}^\num_\kappa \Kk \vec{s}_\kappa \rangle_{L^2(\Omega)^d} \\
        &= \langle \Kk \vec{s}_\kappa, -\ca{L}^\num_\kappa u \rangle_{L^2(\Omega)^d}  + \langle \Kk \vec{s}_\kappa, -\ca{L}^\num_\kappa (u-\Kk \vec{s}_\kappa) \rangle_{L^2(\Omega)^d} .
    \end{aligned}
\end{equation*}
Note that due to (\ref{eqn:numconv}), there is a $C>0$ such that $\|\ca{L}^\num_\kappa\|\leq C$, leading to
\begin{equation*}
     |\langle \Kk \vec{s}_\kappa, -\ca{L}^\num_\kappa (u-\Kk \vec{s}_\kappa) \rangle_{L^2(\Omega)^d}| \leq C\|\Kk \vec{s}_\kappa\|_{L^2(\Omega)^d}\|\Kk \vec{s}_\kappa - u\|_{L^2(\Omega)^d} \xrightarrow[]{\kappa \searrow 0} 0.
\end{equation*}
This leaves us with only having to show
\begin{equation*}
    \langle \Kk \vec{s}_\kappa, -\ca{L}^\num_\kappa u \rangle_{L^2(\Omega)^d} \rightarrow B(u,u) = \langle u , -\ca{L} u\rangle_{L^2(\Omega)^d}.
\end{equation*}
Due the strong convergence of the $\Kk \vec{s}_\kappa$ to $u$ in $L^2(\Omega)^d$, this follows again from (\ref{eqn:numconv}).
\end{proof}
All that is left to prove now is (\ref{eqn:numconv}) and (\ref{eqn:numcoerc}). 
We begin with (\ref{eqn:numconv}). First, the continuous operator can be split as follows (See \cite{Mengesha2014}, Lemma 5),
\begin{equation}
\label{eqn:cont_op_split}
    \begin{aligned}
        &[\ca{L}u](x):= \int_{\Omega} \b{C}(x',x)(u(x')-u(x)) \dd x' \\
        &\mathbf{C}(x',x) := k_1(x',x) + k_2(x',x), \\
        &k_1(x',x) := (\alpha(x)+\alpha(x')) \rho(x'-x)\frac{\chi_{[0,\delta)}(\|x'-x\|)}{\|x'-x\|^2} (x'-x)(x'-x)^\top,\\
        &\begin{aligned}        
            k_2(x',x) :=\int_{\Omega} &\tau(p)\chi_{[0,\delta)}(\|p-x\|)\chi_{[0,\delta)}(\|x'-p\|)\rho(x-p)\rho(p-x')(x-p)(p-x')^\top\\
            -&\tau(x')\chi_{[0,\delta)}(\|x'-x\|)\chi_{[0,\delta)}(\|p-x'\|)\rho(x-x')\rho(x'-p)(x-x')(x'-p)^\top\\
            +&\tau(x)\chi_{[0,\delta)}(\|p-x\|)\chi_{[0,\delta)}(\|x'-x\|)\rho(x-p)\rho(x-x')(x-p)(x-x')^\top\dd p \\
        \end{aligned}\\
    \end{aligned}
\end{equation}
We now split the integral in $k_2$ into three terms corresponding to the summands in the integrand, therefore splitting $k_2(x',x)$ into $k_2^p(x',x)$, $k_2^{x'}(x',x)$ and $k_2^x(x',x)$. Then, we split $\ca{L}$ into four operators by splitting the integral in $\ca{L}$ at the sum defining $\mathbf{C}$, to finally obtain 
\begin{equation*}
\begin{aligned}
    [\ca{L}_1u](x) &:= \int_{\Omega} k_1(x',x)(u(x')-u(x))  \dd x', \\ 
    [\ca{L}_2^{x}u](x) &:= \int_{\Omega} k_2^{x}(x',x)(u(x')-u(x))  \dd x', \\
    [\ca{L}_2^{x'}u](x) &:= \int_{\Omega} k_2^{x'}(x',x)(u(x')-u(x))  \dd x', \\ 
    [\ca{L}_2^{p}u](x) &:= \int_{\Omega} k_2^{p}(x',x)(u(x')-u(x))  \dd x'.
\end{aligned}
\end{equation*}
The numerical operator will now be split in a similar way and the convergence and uniform boundedness claims will then be shown for each operator individually. The proofs for each operator are usually very similar and therefore partially omitted.

For notational simplicity we define $V_j^\supscr{i} := V_j w_{ij}$.\newline

\begin{lemma}[Splitting the Numerical Linear Operator]\label{lemm:OpSplitSb} The operator $\ca{L}^\num_\kappa$ can be split into
\begin{equation*}
    \ca{L}^\num_\kappa = \ca{L}^{\num}_1+\ca{L}^{x,\num}_2+\ca{L}^{x',\num}_2+\ca{L}^{p,\num}_2,
\end{equation*}
with, for $u\in L^2(\Omega)^d$ and $x \in \Omega_i$,
\begin{align}
\label{eqn:defnumconst}
    &[\ca{L}^{\num}_1 u](x) := \frac{1}{V_i}\sum_{j\not=i}  \Big[V_iV_j^\supscr{i}\rho(x_j-x_i)(\alpha^\num(x_i)+\alpha^\num(x_j))\tfrac{\left(x_j-x_i\right)\left(x_j-x_i\right)^\top}{\|x_j-x_i\|\|x_j-x_i\|} ([\Pk u]_j-[\Pk u]_i)\Big] \nonumber\\
    &[\ca{L}^{x,\num}_2u](x) := \frac{1}{V_i}\sum_{\substack{j\not=i \\ m \not= i}}\Big[V_{i}V_{m}^\supscr{i}V_{j}^\supscr{i}\tau^\num(x_{i})\rho(x_{m}-x_{i})\rho(x_{j}-x_{i})\nonumber\\
    &\phantom{=========================}\times\left(x_{m}-x_{i}\right)\left(x_{j}-x_{i}\right)^\top ([\Pk u]_j-[\Pk u]_i)\Big]\nonumber\\
    &[\ca{L}^{x',\num}_2u](x) := -\frac{1}{V_i}\sum_{ \substack{j\not=i\\m\not={j}}}\Big[
    V_{j}V_{i}^\supscr{j}V_{m}^\supscr{j}\tau^\num(x_{j}){\rho(x_{i}-x_{j})\rho(x_{m}-x_{j})}\nonumber\\ &\phantom{=========================}\times\left(x_{i}-x_{j}\right)\left(x_{j}-x_{m}\right)^\top([\Pk u]_j-[\Pk u]_i)\Big]\nonumber\\
    &[\ca{L}^{p,\num}_2u](x) := \frac{1}{V_i}\sum_{\substack{j\not=i\\m\not\in\{i,j\}}}\Big[ V_{m}V_{i}^\supscr{m}V_{j}^\supscr{m}\tau^\num(x_{m}){\rho(x_{i}-x_{m})\rho(x_{j}-x_{m})}\nonumber\\
    &\phantom{=========================}\times\left(x_{i}-x_{m}\right)\left(x_{m}-x_{j}\right)^\top ([\Pk u]_j-[\Pk u]_i)\Big].
\end{align}\end{lemma}
\begin{proof} See Lemma \ref{lemm:SplitSBOp}.\end{proof}

We will now show the convergence and uniform boundedness claims for each of the above operators. Dealing first with the uniform boundedness claims. For this, we need the following well-known bound on the harmonic series obtainable using Abel's partial summation formula (See \cite{INTRANANUMTHEO}, Chapter I.0, Theorem 0.3 and Theorem 0.8).\newline

\begin{lemma}[Growth Bound for the Harmonic Series]\label{lemm:BndHarmSeries} The following inequality is valid for all $N\geq 1$.
\begin{equation*}
    \sum_{n=1}^{N} \frac{1}{n} \leq \log N + 1.
\end{equation*}\qedwhite\end{lemma}

\noindent With this we can prove the following Lemma.\newline

\begin{lemma}[Uniform Bound for the Integrals over the Numerical Weighting Function]\label{lemm:Bndofstep} Let $R>0$ and define $R_\kappa := R + 2\sqrt{d} \kappa$. Then there is a constant $C_R$ depending only on $R$, such that one has for all $\tfrac{\delta}{\sqrt{d}}>\kappa >0$, 
\begin{equation*}
    0 \leq \int_{B_{R_\kappa}(0)} \step^\kappa(x)\dd x < C_R < \infty,
\end{equation*}
where
\begin{equation*}
    \step^\kappa(x) := \begin{cases}
        0 & x \in [-\tfrac{\kappa}{2}, \tfrac{\kappa}{2})^d \\
        \rho(x_i) & x \in B_i \subset B_{R_\kappa}(0) \\
        0 & \text{else.}
    \end{cases}
\end{equation*}\end{lemma}

\begin{proof} The domain $B_{R_\kappa}(0)$ is partitioned into two sets, $A$ and $B$, defined as
\begin{equation*}
    A:= B_{R_\kappa}(0) \cap \bigcup_{i \in Z^\kappa} B_i,\,\, Z^\kappa := \{i\in I^\kappa: x_i \text{ has a zero coordinate}\},\,\,\, B := B_{R_\kappa}(0)\backslash A.
\end{equation*}
We now show
\begin{equation*}
    \begin{aligned}
        &\int_{A} \step^\kappa(x) \dd x \xrightarrow[]{\kappa \searrow 0} 0,\\
        &\int_{B} \step^\kappa(x) \dd x \leq \int_{B_{R + 2\delta}(0)} \frac{1}{\|x\|}\dd x < \infty.
    \end{aligned}
\end{equation*}
Due to the piecewise constant nature of $\step^\kappa(x)$, one can rewrite the integral over $A$ as a sum. For $d=3$, this leads to
\begin{equation*}
    \begin{aligned}
        \int_{A}\step^\kappa(x) \dd x &\stackrel{}{\leq} 6\kappa^3 \sum_{R_\kappa/\kappa\geq i\geq 1} \frac{1}{\kappa i}  + 12\kappa^3 \sum_{R_\kappa/\kappa \geq i\geq 1}\sum_{R_\kappa/\kappa \geq j\geq 1} \frac{1}{\kappa \sqrt{i^2+j^2}}\\
        &\stackrel{}{\leq} 6\kappa^2(1+\tfrac{2(R+2\delta)}{\kappa})\sum_{R_\kappa/\kappa\geq i\geq 1} \frac{1}{i}\\ 
        &\stackrel{}{\leq}  6(\delta/\sqrt{d} + 2(R+2\delta)) \kappa\left(\log\left(\tfrac{R+2\delta}{\kappa}\right) + 1\right) \xrightarrow[]{\kappa \searrow 0 } 0. \\
    \end{aligned}
\end{equation*}
The last inequality was implied by Proposition \ref{lemm:BndHarmSeries}.
Similarly, we obtain for $d=2$,
\begin{equation*}
    \begin{aligned}
        \int_{A}\step^\kappa(x) \dd x &\stackrel{}{\leq} 4\kappa^2 \sum_{R_\kappa/\kappa\geq i\geq 1} \frac{1}{\kappa i} \stackrel{}{\leq} 4 \kappa\left(\log\left(\tfrac{R+2\delta}{\kappa}\right) + 1\right) \xrightarrow[]{\kappa \searrow 0 } 0.
    \end{aligned}
\end{equation*}
 For the integral over $B$, one has for all $x\in B$,
\begin{equation*}
    0 \leq \step^\kappa((x_1,\dots,x_d)^\top) \leq \Bigg( \sum_{d\geq i\geq 1} (x_i - \tfrac{\kappa}{2}\sgn(x_i) )^2\Bigg)^{-1/2}=:F^\kappa((x_1,\dots,x_d)^\top).
\end{equation*}
Then, since $F^\kappa$ is just a modified version of $x \mapsto \frac{1}{\|x\|}$, one has
\begin{equation*}
    0\leq \int_{B} \step^\kappa (x)\dd x\leq \int_{B} F^\kappa(x) \dd x \leq \int_{B_{R+2\delta}(0)} \frac{1}{\|x\|} \dd x < \infty.
\end{equation*}\end{proof}

We will later apply the above bounds together with the following lemma.\newline

\begin{lemma}[Uniform Boundedness for Discrete Finite Difference Operators]\label{lemm:UnifBoundLemma} Let $k_{ij}(\kappa) \in \mathbb{R}^{d\times d}$ be a family of matrices over the index set $I^\kappa \times I^\kappa$ parameterized by $0<\kappa<\tfrac{\delta}{\sqrt{d}}$, then the operators $A_\kappa: L^2(\Omega)^d \rightarrow L^2(\Omega)^d$ defined for $u\in L^2(\Omega)^d$ and $x\in \Omega_i$ by 
\begin{equation*}
    [A_\kappa u](x) = \sum_{j\not=i} V_j k_{ij}(\kappa) ([\Pk u]_j - [\Pk u]_i)
\end{equation*}
are uniformly bounded with respect to $\kappa$, if 
\begin{equation}
    \label{eqn:assunifbnd}
    \exists C >0 : \forall 0<\kappa<\tfrac{\delta}{\sqrt{d}}: \forall i\in I^\kappa: \max\left(\sum_{j\not= i}  \|V_j k_{ij}(\kappa)\|,\sum_{j\not= i}  \|V_j k_{ji}(\kappa)\|\right) \leq C.
\end{equation}\end{lemma}
\begin{proof} The operators $A_\kappa$ are split into a convolutional and a point-wise part, defined for $x\in \Omega_i$ by
\begin{equation*}
    \begin{aligned}
        [A_\kappa^* u](x) &:=  \sum_{j\not=i}V_j k_{ij}(\kappa)[\Pk u]_j   \\
        [A_\kappa^\Sigma u](x) &:= \Bigg(\sum_{j\not=i} - V_j k_{ij}(\kappa) \Bigg) [\Pk u]_i,
    \end{aligned}
\end{equation*}
each of which will be shown to be bounded, beginning with the second,
\begin{equation*}
    \|A^\Sigma_\kappa u\|^2_{L^2} \leq \sum_{i} V_i \Bigg(\sum_{j\not=i}\|V_j k_{ij}(\kappa)\|\Bigg)^2\|[\Pk u]_i\|^2 \leq C^2 \sum_{i} V_i\|[\Pk u]_i\|^2 \leq C^2 \|u\|_{L^2}^2,
\end{equation*}
where we used Jensen's inequality to obtain
\begin{equation}
\label{eqn:JIBOUNDONVOLSUM}
    \sum_{i} V_i\|[\Pk u]_i\|^2 \leq \sum_{i} V_i\left(\frac{1}{V_i}\int_{\Omega_i}\|u(x)\| \dd x\right)^2 \stackrel{}{\leq} \sum_{i} \int_{\Omega_i} \|u(x)\|^2 \dd x = \|u\|^2_{L^2}.
\end{equation}
For the convolutional part, we use Cauchy-Schwartz to obtain
\begin{align*}
        \|A_\kappa^* u\|^2_{L^2}&= \sum_{i} V_i \Bigg(\sum_{j\not=i} V_j\|k_{ij}(\kappa)\|\|[\Pk u]_j\| \Bigg)^2 \\
        &\stackrel{}{\leq}\sum_{i} V_i \Bigg(\sum_{j\not=i} V_j\|k_{ij}(\kappa)\|\Bigg)\Bigg(\sum_{j\not=i} V_j\|k_{ij}(\kappa)\|\|[\Pk u]_j\|^2\Bigg)
        \intertext{applying (\ref{eqn:assunifbnd}), then swapping the sums yields}
        &\leq C \sum_i \sum_{j\not= i} V_iV_j \|k_{ij}(\kappa)\| \|[\Pk u]_j\|^2 \\
        &\leq C \sum_{j} V_j \|[\Pk u]_j\|^2 \sum_{i\not= j} V_i \|k_{ij}(\kappa)\| \\
        &\leq C^2 \|u\|^2_{L^2}.
\end{align*}
The last inequality was obtained using (\ref{eqn:JIBOUNDONVOLSUM}) and (\ref{eqn:assunifbnd}). 
\end{proof}

To apply this lemma, we still need the necessary bounds on the terms involved in the operators in (\ref{eqn:defnumconst}), for which we need the following lemmas.\newline 

\begin{lemma}[Convergence of the Partial Area Weights]\label{lemm:BoundOnDiffOfKappa} For all $i\in I^\kappa$, one has for all $x\in \Omega_i$,
\begin{equation*}
    \begin{aligned}
         \sum_{j\not=i} \int_{\Omega_j} | w_{ij}-\chi_{[0,\delta)}(\|x'-x\|)| \dd x' \stackrel{}{\leq} D^{\delta}_\kappa \xrightarrow[]{\kappa \searrow 0}0
    \end{aligned}
\end{equation*}
for some $D^{\delta}_\kappa>0$ dependent only on $\delta$ and $\kappa$.\end{lemma}
\begin{proof} 
We first obtain
\begin{align*}
     &\ \sum_{j\not=i} \int_{\Omega_j} | w_{ij}-\chi_{[0,\delta)}(\|x'-x\|)| \dd x' \\
     &\leq \sum_{j\not=i} \int_{B_j} | w_{ij}-\chi_{[0,\delta)}(\|x'-x_i\|)| \dd x' + \int_{\mathbb{R}^d} | \chi_{[0,\delta)}(\|x'-x\|)-\chi_{[0,\delta)}(\|x'-x_i\|)| \dd x' \\
     &\leq B_{\delta + \sqrt{d}\kappa}(0) - B_{\delta - \sqrt{d}\kappa}(0) + \sqrt{d}\kappa \delta
\end{align*}
which directly follows from the assumptions on the weights $w_{ij}$ in Assumptions \ref{ass:NumDiscr}, and $\|x-x_i\|\leq \tfrac{1}{2}\sqrt{d}\kappa$.\end{proof}

We can now apply these lemmas to obtain bounds on $\tau^\num$, $\alpha^\num$ and $m^\num$ similarly to those on $\tau$, $m$ and $\alpha$ used in the previous chapter. Note that A5 from Assumptions \ref{ass:NumDiscr} implies that there are $m_0, m_1 > 0$ such that $m_0 \leqq m \leqq m_1$ on $\Omega$. 

\begin{lemma}[Bounds on the Numerical Model Parameters]
\label{lemm:taunumlemma} Let Assumptions \ref{ass:NumDiscr} hold, then there exists a $t_1 >0$ such that for $\kappa>0$ small enough,
\begin{equation*}
    0<\tfrac{1}{2}m_0 \leqq m^\num \leqq 2m_1,\,\,\,\,\,\, |\tau^\num| \leqq t_1,\,\,\,\,\,\, 0<\alpha_0 \leqq \alpha^\num, \alpha \leqq \alpha_1.
\end{equation*}
Let $i\in I^\kappa$ be such that
\begin{equation*}
    \overline{B_i} \cap (\partial \Omega \cup LFS^L(l) \cup LFS^K(k))  =\emptyset,
\end{equation*}
then one has for all $x\in \Omega_i$,
\begin{equation*}
    |\tau(x) - \tau^\num(x_i)| \leq T_\kappa \xrightarrow[]{\kappa \searrow 0} 0,
\end{equation*}
where $T_\kappa$ is independent of $i$.
Similarly, for all $i\in I^\kappa$ such that
\begin{equation*}
    \overline{B}_i \cap \left( \partial \Omega \cup LFS^L(l) \right) = \emptyset,
\end{equation*}
we obtain for all $x \in \Omega_i$,
\begin{equation*}
    |\alpha(x) -\alpha^\num(x_i)| \leq A_\kappa \xrightarrow[]{\kappa \searrow 0} 0,
\end{equation*}
with an $A_\kappa$ independent of $i$.
\end{lemma}
\begin{proof} First, remember that
\begin{equation*}
    0<m_0 \leq m(x) := \int_{\Omega} \omega(x'-x) \|x'-x\|^2 \dd x' = \int_{\Omega} \chi_{[0,\delta)}(\|x'-x\|)\|x'-x\| \dd x'.
\end{equation*}
For any $x\in \Omega_i$ and any $i$, we then have
\begin{equation}
\begin{aligned}
\label{eqn:bndmmnum}
    &|m(x)-m^\num(x_i)|\\
    &\leq \left|\sum_{j} \int_{\Omega_j} \chi_{[0,\delta)}(\|x'-x\|)\|x'-x\| - w_{ij}\|x_j-x_i\| \dd x'\right| \\
    &\leq \sum_{j}\int_{\Omega_j} | \chi_{[0,\delta)}(\|x'-x\|) - w_{ij}| \|x'-x\| \dd x'\\
    &\quad+ \sum_{j}\int_{\Omega_j} w_{ij}|\|x'-x\|-\|x_j-x_i\|| \dd x' \\
    &\stackrel{}{\leq} \diam(\Omega)D^\delta_\kappa + \sqrt{d}\kappa  \vol(\Omega) \xrightarrow{\kappa \searrow 0} 0.
\end{aligned}
\end{equation}
Lemma \ref{lemm:BoundOnDiffOfKappa} was used in the last inequality. The other claims then follow from this bound and Lemma \ref{lemm:LFS}.
\end{proof}

This leads to the uniform boundedness result as follows.\newline

\begin{proposition}[Uniform Boundedness of the Numerical Operators]\label{prop:UnifBound} The operators $\ca{L}^{\num}_1$, $\ca{L}^{x,\num}_2$, $\ca{L}^{x',\num}_2$ and $\ca{L}^{p,\num}_2$ in (\ref{eqn:defnumconst}) are uniformly bounded with respect to any sequence $0<\kappa_n\leq\tfrac{\delta}{\sqrt{d}}$ with $\kappa_n \searrow 0$.\end{proposition}

\begin{proof} We apply Lemma \ref{lemm:UnifBoundLemma}, so that we only have to show (\ref{eqn:assunifbnd}) for each operator, which follows from the symmetry of the kernels and Lemma \ref{lemm:Bndofstep}, yielding for $\ca{L}^\num_1$ and small enough $0<\kappa$,
\begin{align*}
    &\sum_{j\not=i} V_jw_{ij}\rho(x_j-x_i)|\alpha^\num(x_i)+\alpha^\num(x_j)|\stackrel{}{\leq} 2\alpha_1 C_{\delta}, 
    \intertext{for $\ca{L}^{x,\num}_2$,}
    &\sum_{j\not=i} \sum_{m\not=i} V_j V_m |\tau^\num(x_i)|w_{ij}w_{im} \leq t_1 \vol(\Omega)^2,  
    \intertext{for $\ca{L}^{x',\num}_2$,}
    &\sum_{j\not=i} \sum_{m\not=i} V_j V_m |\tau^\num(x_j)|w_{ji}w_{jm} \leq t_1 \vol(\Omega)^2, 
    \intertext{for $\ca{L}^{p,\num}_2$,}
    &\sum_{j\not=i} \sum_{m\not=\{i,j\}} V_j V_m |\tau^\num(x_m)|w_{mj}w_{mi} \leq t_1 \vol(\Omega)^2.
\end{align*}
\end{proof}

With this uniform boundedness result and the boundedness of $\ca{L}$, we only need to show the pointwise convergence of $\ca{L}^\num_\kappa$ to $\ca{L}$ on a dense subspace of $L^2(\Omega)^d$ such as $C^{0,1}(\Omega)^d$. For this we need the following inequalities.

\begin{lemma}[Auxiliary Bounds]\label{lemm:auxbndlemma} The following estimates hold for all $i,j \in I^\kappa$, with $i\not=j$.
\vspace{-0.25cm}
\begin{itemize}
    \item Let $x'\in\Omega_j$, $x\in \Omega_i$, then 
\begin{equation*}
    \left\|\rho(x'-x)(x'-x)-\rho(x_j-x_i)(x_j-x_i)\right\| \leq \min\left(\frac{2\sqrt{d}\kappa}{\|x'-x\|}, \frac{2\sqrt{d}\kappa}{\|x_j-x_i\|} \right).
\end{equation*}
\item For all $x\in \Omega_i$,
\begin{equation*}
    \sum_{j\not= i}\int_{\Omega_j} \left|\rho(x'-x)-\rho(x_j-x_i)\right|\|x'-x\| \dd x' \leq \sqrt{d}\kappa C_{\diam(\Omega)}.
\end{equation*}
\end{itemize}\end{lemma}
\begin{proof}
For the second claim, note that $\rho(\zeta) = \tfrac{1}{\|\zeta\|}$, leading to
\begin{equation*}
    \begin{aligned}
        &\sum_{j\not= i}\int_{\Omega_j} \left|\frac{1}{\|x'-x\|}-\frac{1}{\|x_j-x_i\|}\right|\|x'-x\| \dd x' \\
        &\leq \sum_{j\not= i}\int_{\Omega_j} \frac{\left|\|x'-x\|-\|x_j-x_i\|\right|}{\|x'-x\|\|x_j-x_i\|}\|x'-x\| \dd x' \\
        &\leq \sum_{j\not= i}\int_{\Omega_j} \frac{\|x'-x-x_j+x_i\|}{\|x'-x\|\|x_j-x_i\|}\|x'-x\| \dd x' \\
        &\leq \sqrt{d}\kappa \sum_{j\not= i}\int_{\Omega_j} \frac{1}{\|x_j-x_i\|} \dd x' \\
        &\leq \sqrt{d}\kappa C_{\diam(\Omega)}.
    \end{aligned}
\end{equation*}
The last inequality above made use of Lemma \ref{lemm:Bndofstep}.

The first claim follows similarly.\end{proof}

We now prove the point-wise convergence of the numerical operators to their continuous counterparts. 

\begin{proposition}[Pointwise Convergence on $C^{0,1}(\Omega)^d$]\label{prop:ConvLemma} Let Assumptions \ref{ass:NumDiscr} hold. The operators $\ca{L}^\num_1$, $ \ca{L}^{x,\num}_2$, $\ca{L}^{x',\num}_2$ and $\ca{L}^{p,\num}_2$ converge pointwise to their continuous equivalents $\ca{L}_1$, $ \ca{L}^{x}_2$, $\ca{L}^{x'}_2$ and $\ca{L}^{p}_2$ on $C^{0,1}(\Omega)^d$. \end{proposition}

\begin{proof} Using the triangle inequality, one obtains
\begin{equation*}
    \|(A_\kappa - A)v\|_{L^2(\Omega)^d} \leq \|(A_\kappa - \ProjP \circ A)v\|_{L^2(\Omega)^d} + \underbrace{\|(\idL - \ProjP) A v\|_{L^2(\Omega)^d}}_{\rightarrow 0}
\end{equation*}
for any operators $A$ and $A_\kappa$, so that we only need to show the convergence of the first term on the right hand side for all ${u}\in C^{0,1}(\Omega)^d$. Therefore, let $u$ be Lipschitz-continuous with Lipschitz constant $L>0$. 

For $x\in \Omega_i$, we first get
\begin{equation}
\label{eqn:op1convdiff}
    \begin{aligned}
        &[(\ca{L}_1^\num - \ProjP\circ\ca{L}_1)u](x) \\
        &= \frac{1}{V_i}\int_{\Omega_i}\sum_{j\not=i}\int_{\Omega_j}\Bigg(w_{ij} \rho(x_i-x_j)(\alpha^\num(x_i)+\alpha^\num(x_j))\frac{\left(x_j-x_i\right)\left(x_j-x_i\right)^\top}{\|x_j-x_i\|\|x_j-x_i\|} \\
        &\qquad-\chi_{[0,\delta)}(\|x'-x\|) \rho(x'-x)(\alpha(x)+\alpha(x'))\frac{\left(x'-x\right)\left(x'-x\right)^\top}{\|x'-x\|\|x'-x\|}\Bigg) (u(x')- u(x)) \dd x' \dd x \\
        &\quad-\frac{1}{V_i}\int_{\Omega_i}\int_{\Omega_i}\chi_{[0,\delta)}(\|x'-x\|) \rho(x'-x)(\alpha(x)+\alpha(x'))\tfrac{\left(x'-x\right)\left(x'-x\right)^\top}{\|x'-x\|\|x'-x\|} (u(x')- u(x)) \dd x' \dd x. \\
     \end{aligned}
\end{equation}
The latter term satisfies
\begin{equation}
\label{eqn:lattersummands}
     \begin{aligned}
         &\left\|\frac{1}{V_i}\int_{\Omega_i}\int_{\Omega_i}\chi_{[0,\delta)}(\|x'-x\|) \frac{\alpha(x)+\alpha(x')}{\|x'-x\|}\frac{\left(x'-x\right)\left(x'-x\right)^\top}{\|x'-x\|\|x'-x\|} (u(x')- u(x)) \dd x' \dd x\right\|\\
         &\leq L\frac{1}{V_i}\int_{\Omega_i}\int_{\Omega_i}|\alpha(x)+\alpha(x')| \dd x' \dd x \\
         &\leq 2\alpha_1 L \kappa^d \xrightarrow[]{\kappa \searrow 0} 0.
     \end{aligned}
\end{equation}
Before we treat the first term, we define $P_\kappa:= \{i: \overline{B_i} \cap (\partial \Omega \cup LFS(l)\cup LFS(k))= \emptyset \}$. We will now show
\begin{equation*}
    \begin{aligned}
        &\|(\ca{L}_1^\num - \ProjP\circ\ca{L}_1)u\|^2_{L^2} \\
        &\stackrel{{\kappa << 1}}{\leq} \sum_{i\in P\kappa} \int_{\Omega_i} \underbrace{\|[(\ca{L}_1^\num - \ProjP\circ\ca{L}_1)u](x)\|^2}_{\xRightarrow[\text{unif}]{*} 0}\dd x\\
        &\qquad+\underbrace{\left(\sup_{0<\kappa<<1} \max_{i\not\in P_\kappa}\sup_{x\in \Omega_i} \|[(\ca{L}_1^\num - \ProjP\circ\ca{L}_1)u](x)\|^2\right)}_{\stackrel{*}{<} \infty} \underbrace{\vol\left(\bigcup_{i\not\in P_\kappa} B_i \right).}_{\xrightarrow[**]{\kappa\searrow 0} 0}
    \end{aligned}
\end{equation*}
The claims left to prove are marked with a star. The claim marked with $**$ follows from Proposition \ref{prop:ClosedZeroSetApprox}. This strategy will remain roughly the same for each operator. 

Using it, we only have to show the convergence on the $\Omega_i$ corresponding to $i \in P_\kappa$ and show the uniform boundedness on the rest.  The former will be proven by showing a uniform convergence, $\|[(\ca{L}_1^\num - \ProjP\circ\ca{L}_1)u](x)\|\leq C^{\rho, \Omega,\alpha,\tau}_\kappa \xrightarrow[]{\kappa \searrow 0} 0$ for all $x\in \Omega_i$ and $i\in P_\kappa$.

Note that due to the uniform convergence of the second summand in (\ref{eqn:op1convdiff}) as shown in (\ref{eqn:lattersummands}), we only have to show the boundedness and convergence claims for the first.

It follows for $i\in P_\kappa$, from Lemma \ref{lemm:BoundOnDiffOfKappa}, \ref{lemm:auxbndlemma}, \ref{lemm:Bndofstep} and applying Lemma \ref{lemm:LFS} on $\overline{B_i}$ to obtain the Lipschitz continuity of $\alpha_{|B_i}$, 
\begin{align*}        
    &L\frac{1}{V_i}\int_{\Omega_i}\sum_{j\not=i}\int_{\Omega_j}\Bigg\|w_{ij} (\alpha^\num(x_i)+\alpha^\num(x_j))\rho(x_j-x_i)\frac{\left(x_j-x_i\right)\left(x_j-x_i\right)^\top}{\|x_j-x_i\|\|x_j-x_i\|}\\ 
    &\qquad-\chi_{[0,\delta)}(\|x'-x\|)(\alpha(x)+\alpha(x'))\rho(x'-x)\frac{\left(x'-x\right)\left(x'-x\right)^\top}{\|x'-x\|\|x'-x\|}\Bigg\| \|x'- x\| \dd x' \dd x \\
    &\leq L\frac{1}{V_i}\int_{\Omega_i}\sum_{j\not=i}\int_{\Omega_j}\left|w_{ij}-\chi_{[0,\delta)}(\|x'-x\|)\right| \frac{|\alpha(x)+\alpha(x')|}{\|x'-x\|}\|x'-x\|\dd x' \dd x \\
    &\quad+ L\frac{1}{V_i}\int_{\Omega_i}\sum_{j\not=i}\int_{\Omega_j}w_{ij} \frac{|\alpha(x)-\alpha^\num(x_i)+\alpha(x')-\alpha^\num(x_j)|}{\|x'-x\|}\|x'-x\|\dd x' \dd x \\
    &\quad+ L\frac{2}{V_i}\int_{\Omega_i}\sum_{j\not=i}\int_{\Omega_j}w_{ij} \alpha_1\left|\rho(x'-x) - \rho(x_j-x_i)\right|\|x'-x\|\dd x' \dd x \\
    &\quad+ L\frac{4}{V_i}\int_{\Omega_i}\sum_{j\not=i}\int_{\Omega_j}w_{ij}\rho(x_j-x_i)\alpha_1\left\|\frac{x'-x}{\|x'-x\|}-\frac{x_j-x_i}{\|x_j-x_i\|}\right\|\|x'-x\|\dd x' \dd x \\
    &\stackrel{}{\leq} 2L\alpha_1D^{\delta}_\kappa +L2A_\kappa\vol(\Omega)+L\vol\left(\bigcup_{j\not\in P_\kappa} B_j\right)(A_\kappa + 2\alpha_1) \\ 
    &\quad+ 2L\alpha_1 \sqrt{d}\kappa C_{\diam(\Omega)} + L8\alpha_1\sqrt{d}\kappa C_{\diam(\Omega)} \xrightarrow[P\ref{prop:ClosedZeroSetApprox}]{\kappa \searrow 0} 0.
\end{align*}
The boundedness claims follow directly from
\begin{equation*}
     \begin{aligned}
         &L\frac{1}{V_i}\int_{\Omega_i}\sum_{j\not=i}\int_{\Omega_j}\Bigg\|w_{ij}\frac{|\alpha^\num(x_i)+\alpha^\num(x_j)|}{\|x_j-x_i\|}\frac{\left(x_j-x_i\right)\left(x_j-x_i\right)^\top}{\|x_j-x_i\|\|x_j-x_i\|} \\
         &\qquad-\chi_{[0,\delta)}(\|x'-x\|) \frac{|\alpha(x)+\alpha(x')|}{\|x'-x\|}\frac{\left(x'-x\right)\left(x'-x\right)^\top}{\|x'-x\|\|x'-x\|}\Bigg\| \|x'- x\| \dd x' \dd x \\
         &\leq 2L\alpha_1 \diam(\Omega)C_{\diam(\Omega)}+ L\vol(\Omega)2\alpha_1<\infty.
     \end{aligned}  
\end{equation*}
This implies the convergence of $\ca{L}^{\num}_1u$ to $\ca{L} u$ for $u\in C^{0,1}(\Omega)^d$. The convergence of the other operators follow similarly; assuming $i\in P^\kappa$,
\begin{equation*}
    \begin{aligned}
        &[(\ca{L}_2^{x,\num} - \ProjP\circ\ca{L}_2^{x} )u](x) \\
        &=\frac{1}{V_i}\int_{\Omega_i}\sum_{j\not=i}\int_{\Omega_j}\sum_{m\not=i}\int_{\Omega_m}\Bigg(\tau^\num(x_{i})\rho(x_m-x_i)\left(x_{m}-x_{i}\right)\rho(x_j-x_i)\left(x_{j}-x_{i}\right)^\top w_{ij}w_{im}\\
        &\qquad- \tau(x)\rho(p-x)\left(p-x\right)\rho(x'-x)\left(x'-x\right)^\top\chi_{[0,\delta)}(\|x'-x\|)\chi_{[0,\delta)}(\|p-x\|)\Bigg)\\
        &\qquad\qquad\qquad\times(u(x')-u(x)) \dd p \dd x' \dd x \\
        &\quad -\frac{1}{V_i}\int_{\Omega_i}\sum_{j\not=i}\int_{\Omega_j}\int_{\Omega_i} \tau(x)\frac{\left(p-x\right)\left(x'-x\right)^\top}{\|p-x\|\|x'-x\|}\chi_{[0,\delta)}(\|x'-x\|)\\
        &\qquad\qquad\qquad\times\chi_{[0,\delta)}(\|p-x\|)(u(x')-u(x)) \dd p \dd x' \dd x \\
        &\quad -\frac{1}{V_i}\int_{\Omega_i}\int_{\Omega_i}\int_{\Omega} \tau(x)\frac{\left(p-x\right)\left(x'-x\right)^\top}{\|p-x\|\|x'-x\|}\chi_{[0,\delta)}(\|x'-x\|)\\
        &\qquad\qquad\qquad\times\chi_{[0,\delta)}(\|p-x\|)(u(x')-u(x)) \dd p \dd x' \dd x.
    \end{aligned}
\end{equation*}
The last two terms above converge to $0$ as their integration domains vanish as $\kappa \searrow 0$ while the integrand stays bounded. \newline
The convergence for the first term and $i\in P^\kappa$ is given by Lemma \ref{lemm:taunumlemma}, as well as
\begin{equation*}
    \begin{aligned}
        &L\frac{1}{V_i}\int_{\Omega_i}\sum_{j\not=i}\int_{\Omega_j}\sum_{m\not=i}\int_{\Omega_m}|\tau(x) - \tau^\num(x_i)| \|x'-x\| \dd p \dd x' \dd x \\
        &\quad+L\frac{1}{V_i}\int_{\Omega_i}\sum_{j\not=i}\int_{\Omega_j}\sum_{m\not=i}\int_{\Omega_m} |\tau^\num(x_i)|\frac{2\sqrt{d}\kappa}{\|x_m-x_i\|}\|x'-x\| \dd p \dd x' \dd x \\
        &\quad+L\frac{1}{V_i}\int_{\Omega_i}\sum_{j\not=i}\int_{\Omega_j}\sum_{m\not=i}\int_{\Omega_m} |\tau^\num(x_i)|\frac{2\sqrt{d}\kappa}{\|x'-x\|}\|x'-x\| \dd p \dd x' \dd x \\
        &\quad+L\frac{1}{V_i}\int_{\Omega_i}\sum_{j\not=i}\int_{\Omega_j}\sum_{m\not=i}\int_{\Omega_m} |\tau^\num(x_i)| |w_{ij}-\chi_{[0,\delta)}(\|x'-x\|) |\|x'-x\|\dd p \dd x' \dd x \\   
        &\quad+L\frac{1}{V_i}\int_{\Omega_i}\sum_{j\not=i}\int_{\Omega_j}\sum_{m\not=i}\int_{\Omega_m}  |\tau^\num(x_i)|| \chi_{[0,\delta)}(\|p-x\|)- w_{im}|\|x'-x\|\dd p \dd x' \dd x \\
        &\leq LT_\kappa \vol(\Omega)^2\diam(\Omega) +2Lt_1\sqrt{d}\kappa \diam(\Omega)C_{\diam(\Omega)}\vol(\Omega) +2Lt_1\sqrt{d}\kappa \vol(\Omega)^2\\
        &\quad+Lt_1\diam(\Omega)\vol(\Omega)D^{\delta}_\kappa+ Lt_1\diam(\Omega)\vol(\Omega)D^{\delta}_\kappa \kappa \xrightarrow[]{\kappa \searrow 0} 0.
    \end{aligned}
\end{equation*}
The boundedness follows from
\begin{equation*}
    \begin{aligned}
        &\Bigg\|\frac{1}{V_i}\int_{\Omega_i}\sum_{j\not=i}\int_{\Omega_j}\sum_{m\not=i}\int_{\Omega_m}\Bigg(\tau^\num(x_{i})\tfrac{\left(x_{m}-x_{i}\right)\left(x_{j}-x_{i}\right)^\top}{\|x_{m}-x_{i}\|\|x_{j}-x_{i}\|}w_{ij}w_{im}\\
        &\quad- \tau(x)\tfrac{\left(p-x\right)\left(x'-x\right)^\top}{\|p-x\|\|x'-x\|}\chi_{[0,\delta)}(\|x'-x\|)\chi_{[0,\delta)}(\|p-x\|)\Bigg)(u(x')-u(x)) \dd p \dd x' \dd x\Bigg\| \\
        &\leq 2Lt_1\diam(\Omega)\vol(\Omega)^2.
    \end{aligned}
\end{equation*}
This gives the pointwise convergence of the second operator.
The convergence of $\ca{L}_2^{x',\num}$ and $\ca{L}_2^{p,\num}$ follows from splitting the sums into $j\in P_\kappa$ and $j\not\in P_\kappa$ or $m\in P_\kappa$ and $m\not\in P_\kappa$, respectively, and showing the convergence of both parts with similar tools as above. \end{proof}

As described prior, this yields the statement in (\ref{eqn:numconv}).

\begin{corollary}[Uniform Boundedness and Pointwise Convergence]\label{corol:NumConvCoroll} Let Assumptions \ref{ass:NumDiscr} hold. Then the operators $\ca{L}^\num_\kappa$ are uniformly bounded with respect to $\kappa$ and pointwise convergent to $\ca{L}$ on $L^2(\Omega)^d$. \end{corollary}

\begin{proof} Note that $C^\infty_0(\Omega)^d\subset C^{0,1}(\Omega)^d$ is dense in $L^2(\Omega)^d$. Further, the operators $\ca{L}^\num_\kappa$ are uniformly bounded, and $\ca{L}$ is bounded. This implies, together with Proposition \ref{prop:ConvLemma}, that the point-wise convergence holds on all of $L^2(\Omega)^d$. \end{proof}

The only thing now left to prove are the coercitivity claims stated in (\ref{eqn:numcoerc}), for which we require the following result. Its rather lengthy but elementary proof is given in the appendix.\newline

\begin{lemma}[Potential Form of The Bilinear Form]\label{lemm:PotFormLemma} One has for all $\vec{u} \in \vec{\ca{V}}_\kappa$,
\begin{equation*}
    \vec{u}^\top \vec{B}^\num \vec{u} = 2\sum_{i} V_i W_\kappa^\num(\vec{u}, x_i),\vspace{-0.2cm}
\end{equation*}
with 
\begin{equation}
\begin{aligned}
\label{eqn:potformWeqn}
    W_\kappa^\num(\vec{u}, x_i)&=\frac{k(x_i)}{2}\left(\frac{d}{m^\num(x_i)}\sum_{j\not= i}\int_{\Omega_j}w_{ij}\rho(x_{j}-x_{i})\left(x_{j}-x_{i}\right)^\top(\vec{u}_j-\vec{u}_i)  \dd x'\right)^2 \\
    &\quad+\frac{\alpha^\num(x_i)}{2}\sum_{j\not=i} \int_{\Omega_j}w_{ij}\rho(x_j-x_i)\|x_j-x_i\|^2\Bigg(\frac{\left(x_j-x_i\right)^\top(\vec{u}_j-\vec{u}_i)}{\|x_j-x_i\|^2} \\
    &\qquad- \frac{1}{m^\num(x_i)}\sum_{m\not= i}\int_{\Omega_m}w_{im}\rho(x_m-x_i)\left(x_{m}-x_{i}\right)^\top(\vec{u}_m-\vec{u}_i)  \dd p\Bigg)^2 \dd x'
\end{aligned}
\end{equation}\end{lemma}
\begin{proof}See Lemma \ref{lemm:potform}. \end{proof}

The form given in (\ref{eqn:potformWeqn}) is the numerical equivalent to the continuous potential, and just like for the continuous potential, one can derive a corresponding lower estimate as follows (compare with \cite{Mengesha2014}, Proposition 1).

\begin{lemma}[Helper Term Lower Bound]\label{lemm:IneqDiscrPot} There is a $c'>0$ such that for all $\kappa>0$ small enough, one has for all $\vec{u} \in \vec{\ca{V}}_\kappa$,
\begin{equation}
\label{eqn:numlowerbound1}
    \begin{aligned}
         &c' \sum_{i} \int_{\Omega_i} \sum_{j\not=i} \int_{\Omega_j} w_{ij} \rho(x_j-x_i) \|x_j-x_i\|^2\left(\frac{\left(x_j-x_i\right)^\top (\vec{u}_j - \vec{u}_i)}{\|x_j-x_i\|^2}\right)^2 \dd x' \dd x\\
         &\leq \sum_{i} V_i W_\kappa^\num(\vec{u}, x_i).
    \end{aligned}
\end{equation}\end{lemma}
\begin{proof}
\begin{equation*}
    \begin{aligned}
        &\sum_{i} \int_{\Omega_i} \sum_{j\not=i} \int_{\Omega_j} w_{ij} \rho(x_j-x_i)\|x_j-x_i\|^2\left(\frac{\left(x_j-x_i\right)^\top (\vec{u}_j - \vec{u}_i)}{\|x_j-x_i\|^2}\right)^2 \dd x' \dd x \\
        &=\sum_{i} \int_{\Omega_i} \sum_{j\not=i} \int_{\Omega_j} w_{ij} \rho(x_j-x_i)\|x_j-x_i\|^2\Bigg(\frac{\left(x_j-x_i\right)^\top (\vec{u}_j - \vec{u}_i)}{\|x_j-x_i\|^2}\\
        &\qquad- \frac{1}{m^\num(x_i)}\sum_{m\not= i}\int_{\Omega_m}w_{im}\rho(x_m-x_i)\left(x_{m}-x_{i}\right)^\top(\vec{u}_m-\vec{u}_i)  \dd p \Bigg)^2 \dd x' \\
        &\quad +\frac{m^\num(x_i)}{d^2}\left(\frac{d}{m^\num(x_i)}\sum_{m\not= i}\int_{\Omega_m}w_{im}\rho(x_m-x_i)\left(x_{m}-x_{i}\right)^\top(\vec{u}_m-\vec{u}_i) \dd p\right)^2\dd x\\
    \end{aligned}
\end{equation*}
\begin{equation*}
    \begin{aligned}
        &\leq\sum_{i} \int_{\Omega_i} \frac{1}{\alpha_0}\alpha^\num(x_i)\sum_{j\not=i} \int_{\Omega_j} w_{ij} \rho(x_j-x_i)\|x_j-x_i\|^2\Bigg(\frac{\left(x_j-x_i\right)^\top (\vec{u}_j - \vec{u}_i)}{\|x_j-x_i\|^2}\\
        &\qquad\quad- \frac{1}{m^\num(x_i)}\sum_{m\not= i}\int_{\Omega_m}w_{im}\rho(x_{m}-x_{i})\left(x_{m}-x_{i}\right)^\top(\vec{u}_m-\vec{u}_i)  \dd p \Bigg)^2 \dd x' \\
        &\qquad+ \frac{m_1}{d^2k_0}k(x_i)\left(\frac{d}{m^\num(x_i)}\sum_{m\not= i}\int_{\Omega_m}w_{im}\rho(x_{m}-x_{i})\left(x_{m}-x_{i}\right)^\top(\vec{u}_m-\vec{u}_i) \dd p \right)^2\dd x \\
        &\stackrel{}{\leq}\left(\tfrac{1}{\alpha_0} + \tfrac{2m_1}{d^2 k_0}\right)2 \sum_{i} V_iW_\kappa^\num(\vec{u}, x_i).
    \end{aligned}
\end{equation*}
The last inequality follows from Lemma \ref{lemm:taunumlemma}.
\end{proof}

To finally obtain (\ref{eqn:numcoerc}), the left-hand side of the numerical lower bound given in (\ref{eqn:numlowerbound1}) is compared to its continuous counterpart. Using the continuous coercitivity bound then results in a uniform numerical bound for small enough $\kappa$.

 \begin{lemma}[Uniform Coercitivity of the Helper Term]\label{lemm:CoercLastStep} There exists a $c'' >0$ such that for small enough $\kappa >0$, one has
\begin{equation*}
    c'' \|\Kk \vec{u}\|^2_{L^2(\Omega)^d} \leq \sum_{i} \int_{\Omega_i} \sum_{j\not=i} \int_{\Omega_j} w_{ij}\rho(x_j-x_i) \|x_j-x_i\|^2\left(\frac{\left(x_j-x_i\right)^\top (\vec{u}_j - \vec{u}_i)}{\|x_j-x_i\|^2}\right)^2 \dd x' \dd x
\end{equation*}
for all $\vec{u}\in \vec{\ca{V}}_\kappa$. \end{lemma}

\begin{proof}
First, for $\kappa <<1$,
\begin{equation*}
    \begin{aligned} 
        &\frac{1}{(\diam(\Omega)+\sqrt{d})^4}\sum_{i} \int_{\Omega_i} \sum_{j\not=i} \int_{\Omega_j} w_{ij}\|x_j-x_i\|\left(\left(x_j-x_i\right)^\top (\vec{u}_j - \vec{u}_i)\right)^2 \dd x' \dd x \\
        &\leq\sum_{i} \int_{\Omega_i} \sum_{j\not=i} \int_{\Omega_j} w_{ij} \|x_j-x_i\|\left(\frac{\left(x_j-x_i\right)^\top (\vec{u}_j - \vec{u}_i)}{\|x_j-x_i\|^2}\right)^2 \dd x' \dd x.
    \end{aligned}
\end{equation*}
This reduces the proof to showing the desired lower bound holds for the first term.\newline
Using \cite{Mengesha2014}, Proposition 2, in the first inequality, one has for $\vec{u}\in \vec{\ca{V}}_\kappa$,
\begin{equation}
\label{eqn:contbndfr}
    \begin{aligned}
        &c\|\Kk \vec{u}\|^2_{L^2} \\
        &\stackrel{}{\leq} \int_{\Omega} \int_{\Omega}  \chi_{[0,\delta)}(\|x'-x\|)\rho(x'-x)\|x'-x\|^2\left(\left(x'-x\right)^\top ([\Kk \vec{u}](x')- [\Kk \vec{u}](x))\right)^2 \dd x' \dd x \\
        &=  \sum_{i} \int_{\Omega_i} \sum_{j\not= i} \int_{\Omega_j} \chi_{[0,\delta)}(\|x'-x\|)\rho(x'-x)\|x'-x\|^2\left(\left(x'-x\right)^\top (\vec{u}_j- \vec{u}_i)\right)^2 \dd x' \dd x.
   \end{aligned}
\end{equation} 
To obtain a bound for the numerical term that is independent of $\kappa$ for small enough $\kappa<<1$, we compare it to the continuous term above, working towards inequality (\ref{eqn:bndondiffdepu}) below. To this end, we have
\begin{align}
    &\Bigg|\sum_{i}\int_{\Omega_i} \sum_{j\not=i} \int_{\Omega_j}w_{ij}\|x_j-x_i\|\left(\left(x_j-x_i\right)^\top (\vec{u}_j - \vec{u}_i)\right)^2 \dd x' \dd x\nonumber\\
    &\quad-\sum_{i} \int_{\Omega_i} \sum_{j\not= i} \int_{\Omega_j} \chi_{[0,\delta)}(\|x'-x\|)\|x'-x\|\left(\left(x'-x\right)^\top (\vec{u}_j- \vec{u}_i)\right)^2 \dd x' \dd x\Bigg| \nonumber\\
    &\leq \Bigg| \sum_{i}\int_{\Omega_i} \sum_{j\not=i} \int_{\Omega_j} w_{ij}\|x_j-x_i\|\left(\left(x_j-x_i\right)^\top (\vec{u}_j - \vec{u}_i)\right)^2 \nonumber \\
    &\qquad\quad -\chi_{[0,\delta)}(\|x'-x\|)\|x'-x\|\left(\left(x'-x\right)^\top (\vec{u}_j- \vec{u}_i)\right)^2 \dd x' \dd x \Bigg|\nonumber\\
    &\leq\sum_{i}\int_{\Omega_i} \sum_{j\not=i} \int_{\Omega_j}  |\chi_{[0,\delta)}(\|x'-x\|)-w_{ij}|\|x'-x\|\left(\left(x'-x\right)^\top (\vec{u}_j- \vec{u}_i)\right)^2 \dd x' \dd x\nonumber\\
    &\quad+\sum_{i}\int_{\Omega_i} \sum_{j\not=i} \int_{\Omega_j} w_{ij}|\|x'-x\|-\|x_j-x_i\||\left(\left(x'-x\right)^\top (\vec{u}_j- \vec{u}_i)\right)^2  \dd x' \dd x\nonumber \\
    &\quad+\sum_{i}\int_{\Omega_i} \sum_{j\not=i} \int_{\Omega_j}  w_{ij}\|x_j-x_i\|\Bigg|\left(\left(x'-x\right)^\top (\vec{u}_j- \vec{u}_i)\right)^2\nonumber\\
    &\phantom{=======================}-\left(\left(x_j-x_i\right)^\top (\vec{u}_j- \vec{u}_i)\right)^2\Bigg|  \dd x' \dd x\nonumber \\
    &\leq \sum_{i}\int_{\Omega_i} \sum_{j\not=i} \int_{\Omega_j}  |\chi_{[0,\delta)}(\|x'-x\|)-w_{ij}|\|x'-x\|^3\nonumber\\
    &\phantom{=======================}\times(\|[\Kk \vec{u}](x')\|+\|[\Kk \vec{u}](x)\|)^2 \dd x' \dd x\nonumber\\
    &\quad+\sum_{i}\int_{\Omega_i} \sum_{j\not=i} \int_{\Omega_j} \|x'-x-x_j+x_i\|\|x'-x\|^2(\|[\Kk \vec{u}](x')\|+\|[\Kk \vec{u}](x)\|)^2  \dd x' \dd x \nonumber\\
    &\quad+\sum_{i}\int_{\Omega_i} \sum_{j\not=i} \int_{\Omega_j} \|x_j-x_i\|\|x_j-x_i+x'-x\|\left\|x_j-x_i-x'+x\right\|\nonumber\\
    &\phantom{=======================}\times(\|[\Kk \vec{u}](x')\|+\|[\Kk \vec{u}](x)\|)^2  \dd x' \dd x \nonumber\\
    &\leq \diam(\Omega)^3\sum_{i}\int_{\Omega_i} \sum_{j\not=i} \int_{\Omega_j}  |\chi_{[0,\delta)}(\|x'-x\|)-w_{ij}|\nonumber\\
    &\phantom{=======================}\times(\|[\Kk \vec{u}](x')\|+\|[\Kk \vec{u}](x)\|)^2 \dd x' \dd x\nonumber\\
    &\quad+ 3\sqrt{d}(\diam(\Omega)+\sqrt{d})^2\kappa \sum_{i}\int_{\Omega_i} \sum_{j\not=i} \int_{\Omega_j}(\|[\Kk\vec{u}](x')\|+\|[\Kk \vec{u}](x)\|)^2  \dd x' \dd x.\label{eqn:lastlineproofcoerc}
\end{align}
We now use
\begin{equation*}
\begin{aligned}
    &\sum_{i}\int_{\Omega_i} \sum_{j\not=i} \int_{\Omega_j}(\|[\Kk \vec{u}](x')\|+\|[\Kk \vec{u}](x)\|)^2  \dd x' \dd x \\
    &\leq \sum_{i}\int_{\Omega_i} \sum_{j} \int_{\Omega_j}\|[\Kk \vec{u}](x')\|^2 + \|[\Kk \vec{u}](x)\|^2 + 2\|[\Kk \vec{u}](x')\|\|[\Kk \vec{u}](x)\| \dd x' \dd x \\
    &\leq 2 \vol(\Omega)\|\Kk\vec{u}\|^2_{L^2(\Omega)^d} + 2\left(\int_{\Omega} \|[\Kk \vec{u}](x)\| \dd x\right)^2\stackrel{}{\leq} 4\vol(\Omega)\|\Kk \vec{u}\|^2_{L^2(\Omega)^d}.
\end{aligned}
\end{equation*}
Jensen's Inequality was used to obtain the last inequality. 
In order to deal with the first summand of the last expression in (\ref{eqn:lastlineproofcoerc}), for $\kappa < \frac{\delta}{\sqrt{d}}$, we use

\begin{align*}
    &\sum_{i}\sum_{j\not=i} \left(\int_{\Omega_i} \int_{\Omega_j}  |\chi_{[0,\delta)}(\|x'-x\|)-w_{ij}| \dd x' \dd x\right) (\|\vec{u}_j\|+\|\vec{u}_i\|)^2 \\
    &=\sum_{i}\|\vec{u}_i\|^2\sum_{j\not=i} \left(\int_{\Omega_i} \int_{\Omega_j}  |\chi_{[0,\delta)}(\|x'-x\|)-w_{ij}| \dd x' \dd x\right) \\
    &\quad+\sum_{j}\|\vec{u}_j\|^2\sum_{i\not=j} \left(\int_{\Omega_i} \int_{\Omega_j}  |\chi_{[0,\delta)}(\|x'-x\|)-w_{ij}| \dd x' \dd x\right)\\
    &\quad+2\sum_{j}\sum_{i\not=j} \|\vec{u}_j\|\|\vec{u}_i\|\left(\int_{\Omega_i} \int_{\Omega_j}  |\chi_{[0,\delta)}(\|x'-x\|)-w_{ij}| \dd x' \dd x\right)
    \intertext{using Lemma \ref{lemm:BoundOnDiffOfKappa} we then obtain}
    &\stackrel{}{=}  \sum_{i} V_i \|\vec{u}_i\|^2 D^{\delta}_\kappa  +\sum_{j} V_j \|\vec{u}_j\|^2 D^{\delta}_\kappa\\
    &\quad+2\sum_{j}\sum_{i\not=j} \|\vec{u}_j\|\|\vec{u}_i\|\left(\int_{\Omega_i} \int_{\Omega_j}  |\chi_{[0,\delta)}(\|x'-x\|)-w_{ij}| \dd x' \dd x\right)\\
    &\leq 2D^{\delta}_\kappa \|\Kk \vec{u}\|^2_{L^2(\Omega)^d} \\
    &\quad+2\sum_{j}\sum_{i\not=j}\|\vec{u}_j\|\|\vec{u}_i\|\left(\int_{\Omega_i} \int_{\Omega_j}  |\chi_{[0,\delta)}(\|x'-x\|)-w_{ij}| \dd x' \dd x\right)\\
\end{align*}
where finally,
\begin{align*}                              
    &\sum_{j}\sum_{i\not=j}\|\vec{u}_j\|\|\vec{u}_i\|\Bigg(\int_{\Omega_i} \int_{\Omega_j}  |\chi_{[0,\delta)}(\|x'-x\|)-w_{ij}| \dd x' \dd x\Bigg) \\    &=\sum_{j}\|\vec{u}_j\|\Bigg(\int_{\Omega_j} \sum_{i\not=j}\int_{\Omega_i} \|[\Kk \vec{u}](x)\| \big|\chi_{[0,\delta)}(\|x'-x\|)-w_{ij}\big| \dd x \dd x'\Bigg)     \intertext{applying Hölder's Inequality yields with Lemma \ref{lemm:BoundOnDiffOfKappa},}
    &\stackrel{}{\leq} \sum_{j} \|\vec{u}_j\|\Bigg(\int_{\Omega_j}\sum_{i\not=j}\int_{\Omega_i} |\chi_{[0,\delta)}(\|x'-x\|)-w_{ij}| \dd x \dd x'\Bigg)^{1/2}\\
    &\phantom{====\,\,}\times\Bigg(\int_{\Omega_j}\int_{\Omega} \|[\Kk \vec{u}](x)\|^2 \dd x\dd x'\Bigg)^{1/2} \\
    &\leq \sum_{j}\|\vec{u}_j\| (V_jD^{\delta}_\kappa)^{1/2} (V_j \|\Kk \vec{u}\|_{L^2(\Omega)^d}^2)^{1/2} \\
    &\leq \sum_{j} V_j \|\vec{u}_j\| \|\Kk \vec{u}\|_{L^2(\Omega)^d} \sqrt{D^{\delta}_\kappa} \\
    &\leq \sqrt{D^{\delta}_\kappa}\|\Kk \vec{u}\|_{L^2(\Omega)^d} \|\Kk \vec{u}\|_{L^1(\Omega)^d} \leq \sqrt{D^{\delta}_\kappa\vol(\Omega)} \|\Kk \vec{u}\|_{L^2(\Omega)^d}^2.
    \end{align*}
In summary,
\begin{equation}
\begin{aligned}
\label{eqn:bndondiffdepu}
    &\Bigg|\sum_{i}\int_{\Omega_i} \sum_{j\not=i} \int_{\Omega_j}w_{ij}\|x_j-x_i\|\left(\left(x_j-x_i\right)^\top (\vec{u}_j - \vec{u}_i)\right)^2 \dd x' \dd x\\
    &\phantom{==}-\sum_{i} \int_{\Omega_i} \sum_{j\not= i} \int_{\Omega_j} \chi_{[0,\delta)}(\|x'-x\|)\|x'-x\|\left(\left(x'-x\right)^\top (\vec{u}_j- \vec{u}_i)\right)^2 \dd x' \dd x\Bigg| \\
    &\leq C^{\Omega,\delta,\kappa} \|\Kk \vec{u}\|^2_{L^2(\Omega)^d}.
\end{aligned}
\end{equation}
For some $ C^{\Omega,\delta,\kappa} \xrightarrow[]{\kappa \searrow 0} 0$. This implies 
\begin{equation}
\label{eqn:lastlinecoercbound}
\begin{aligned}
    &\sum_{i}\int_{\Omega_i} \sum_{j\not=i} \int_{\Omega_j}w_{ij}\|x_j-x_i\|\left(\left(x_j-x_i\right)^\top (\vec{u}_j - \vec{u}_i)\right)^2 \dd x' \dd x \\
    &\geq \Bigg|\int_{\Omega} \int_{\Omega}  \chi_{[0,\delta)}(\|x'-x\|)\|x'-x\|\left(\left(x'-x\right)^\top ([\Kk \vec{u}](x')- [\Kk \vec{u}](x))\right)^2 \dd x' \dd x \\
    &\quad-\Bigg|\sum_{i}\int_{\Omega_i} \sum_{j\not=i} \int_{\Omega_j}w_{ij}\|x_j-x_i\|\left(\left(x_j-x_i\right)^\top (\vec{u}_j - \vec{u}_i)\right)^2 \dd x' \dd x\\
    &\qquad-\sum_{i} \int_{\Omega_i} \sum_{j\not= i} \int_{\Omega_j} \chi_{[0,\delta)}(\|x'-x\|)\|x'-x\|\left(\left(x'-x\right)^\top (\vec{u}_j- \vec{u}_i)\right)^2 \dd x' \dd x\Bigg|\,\Bigg|.
\end{aligned}
\end{equation}
Choosing $\kappa$ such that ${C^{\Omega,\delta,\kappa'}}^2 < q^2c^2$ for all $0<\kappa'<\kappa$ for some $0<q<1$, allows using (\ref{eqn:contbndfr}) to estimate the last line of (\ref{eqn:lastlinecoercbound}) above as follows,
\begin{equation*}
    \begin{aligned}
    &\geq\int_{\Omega} \int_{\Omega}  \chi_{[0,\delta)}(\|x'-x\|)\|x'-x\|\left(\left(x'-x\right)^\top ([\Kk \vec{u}](x')- [\Kk \vec{u}](x))\right)^2 \dd x' \dd x \\
        &\phantom{=}-\Bigg|\sum_{i}\int_{\Omega_i} \sum_{j\not=i} \int_{\Omega_j}w_{ij}\|x_j-x_i\|\left(\left(x_j-x_i\right)^\top (\vec{u}_j - \vec{u}_i)\right)^2 \dd x' \dd x\\
    &\phantom{==}-\sum_{i} \int_{\Omega_i} \sum_{j\not= i} \int_{\Omega_j} \chi_{[0,\delta)}(\|x'-x\|)\|x'-x\|\left(\left(x'-x\right)^\top (\vec{u}_j- \vec{u}_i)\right)^2 \dd x' \dd x\Bigg| \\
    &\geq (c  - C^{\Omega,d,\delta,\kappa'}) \|\Kk \vec{u}\|^2_{L^2(\Omega)^d}\geq (1-q)c \|\Kk \vec{u}\|^2_{L^2(\Omega)^d}.
    \end{aligned}
\end{equation*}
\end{proof}

\begin{corollary}[Uniform Coercitivity of the Numerical Operators]\label{corol:CoercResult} Let Assumptions \ref{ass:NumDiscr} hold. For small enough $1 >>\kappa >0$ and for all $\vec{u}\in \vec{\ca{V}}_\kappa$, we have
\begin{equation*}
    c^\num \|\Kk \vec{u}\|^2_{L^2} \leq \vec{u}^\top \vec{B}^\num \vec{u},
\end{equation*}
for some $c^\num >0$.\end{corollary}

\begin{proof} This is obtained by combining Lemma \ref{lemm:PotFormLemma}, Lemma \ref{lemm:IneqDiscrPot} and finally Lemma \ref{lemm:CoercLastStep}.
\end{proof}

\begin{remark}\label{rm:ClosingRemarkKernel} The concrete choice for the weighting function $\omega$ as asserted by A1 in Assumptions \ref{ass:NumDiscr}, was used to show uniform bound on the discrete integrals over the weighting function in Lemma \ref{lemm:Bndofstep}, the bounds and convergence results in Lemma \ref{lemm:taunumlemma} and the bounds in \ref{lemm:auxbndlemma} needed for the boundedness in Proposition \ref{prop:UnifBound} and convergence in Proposition \ref{prop:ConvLemma}. It was also used for the coercitivity claims in Lemma \ref{lemm:CoercLastStep}. However, Lemma \ref{lemm:OpSplitSb}, \ref{lemm:IneqDiscrPot} and \ref{lemm:PotFormLemma} can be analogously formulated for other $\omega$. \newline Therefore, it is evident that the convergence can also be shown for other choices of $\omega$, as long as suitable boundedness, convergence and coercitivity results can be shown. Most notably, if $\omega = \chi_{[0,\delta)}(\|x'-x\|) \rho(x'-x)$ with some $\rho \in C^0(\mathbb{R}^d)$ fulfilling Assumptions \ref{ass:GenAssump}, then $\rho$ is both bounded and uniformly continuous on every compact neighborhood of $0$. This significantly simplifies corresponding results in Lemma \ref{lemm:Bndofstep}, \ref{lemm:taunumlemma}, as well as Lemma \ref{lemm:auxbndlemma}, Proposition \ref{prop:ConvLemma} and Lemma \ref{lemm:CoercLastStep}, by making use of 
\begin{equation*}
    \|\rho\|_{L^\infty} <\infty \text{ \ and \ } \|R\|_{L^\infty(\Omega\times \Omega)} \xrightarrow{\kappa \searrow 0} 0  
\end{equation*}
for
\begin{equation*}
    R(x',x) := \sum_{\substack{i\\ j\not=i}} \chi_{\Omega_i}(x)\chi_{\Omega_j}(x') \rho(x_i - x_j) - \rho(x'-x).
\end{equation*}
Most importantly,  Proposition 2 in \cite{Mengesha2014} used in Lemma \ref{lemm:CoercLastStep} Equation (\ref{eqn:contbndfr}) stays valid, as well as Lemma 5 in \cite{Mengesha2014} used in (\ref{eqn:cont_op_split}) to split the continuous operator. \newline
In other words, the singular weighting function assumed in Chapter 3 using $\rho(\zeta) = \frac{1}{\|\zeta\|}\not\in C^0(\mathbb{R}^d)$ represents a  more difficult, however practically relevant special case. Therefore it is not surprising that the proof above can easily be simplified to yield corresponding convergence results for more regular weighting functions defined by a $\rho \in C^0(\mathbb{R}^d)$. In particular, one can obtain similar convergence results for the following choices by adapting the steps in Chapter 3.  
\begin{enumerate}
    \item Constant functions: $\rho\equiv 1$ yielding $\omega(x'-x) = \chi_{[0,\delta)}(\|x'-x\|)$.
    \item Conical functions for some $C>0$ (See \cite{BobarFlorConv1DPD}, \cite{Ha2010}) 
    \begin{equation*}
        \omega(\zeta) = \chi_{[0,\delta)}(\|\zeta\|)\rho(\zeta) := C\max\Big(0,1 - \frac{\|\zeta\|}{\delta}\Big).
    \end{equation*}
    \item Polynomials (See \cite{Seleson2018}, \cite{ImprOnePointQuadrData}, \cite{SELESON20162432})
    \begin{equation}
    \label{eqn:polykernel}
        \omega(\zeta) =\chi_{[0,\delta)}(\|\zeta\|) \rho(\zeta) :=\chi_{[0,\delta)}(\|\zeta\|) \frac{p_n(\|\zeta\|)}{\|\zeta\|^\alpha}  
    \end{equation}
    for some polynomial $p_n: [0,\delta] \rightarrow \mathbb{R}_{\geq 0}$ such that Assumptions \ref{ass:GenAssump} are fulfilled and $\alpha \in \{ 0, 1\}$. The proof for $\alpha=0$ follows from $p_n(\|\cdot\|) \in C^0(\mathbb{R}^d)$. If $\alpha=1$, then
    \begin{equation*}
        \frac{p_n(\|\zeta\|)}{\|\zeta\|^\alpha} = \frac{C}{\|\zeta\|} + p'_n(\|\zeta\|) 
    \end{equation*}
    for another polynomial $p'_n(\|\zeta\|)$. Since the terms besides $m^\num$,$ \alpha^\num$ and $\tau^\num$ in Lemma \ref{lemm:CoercLastStep}, \ref{lemm:auxbndlemma}, Proposition \ref{prop:ConvLemma} and the $m(x)-m^\num(x_i)$ term in Lemma \ref{lemm:taunumlemma} are (sub)linear in $\rho$, this allows combining the proof using $\rho(\|\zeta\|) = \frac{1}{\|\zeta\|}$ with the proof using $\rho(\|\zeta\|) = p'_n(\|\zeta\|)$ to obtain the convergence for the case in (\ref{eqn:polykernel}) using $\alpha=1$ as well. Note that Lemma \ref{lemm:taunumlemma} handles the convergence of the $m^\num$,$ \alpha^\num$ and $\tau^\num$ terms.
\end{enumerate}
\end{remark}

%% file: cnt/num_mot/nummot.tex
\begin{remark}\label{rm:WhyWeights} As a short demonstration of the advantage of including these weights, we compared the convergence of the solutions for a simple 2d bar problem under tensile loading using two sets of weights
\begin{equation}
\label{eqn:defweightsPA}
    w^{\text{FA}}_{ij} :=  \chi_{[0,\delta)}(\|x_j-x_i\|) \text{ and }  w^{\mathrm{PAAC}}_{ij} := \vol(\Omega_j \cap B_{\delta}(x_i))/V_j.
\end{equation}
The FA weights correspond to keeping the characteristic function without using weights in the discretized terms (\ref{eqn:defnumconstparts}). The PAAC terms correspond to using the exact area of the intersection of the horizon with the discrete elements $\Omega_i$. Formulas for the $w_{ij}^{\text{FA}}$ weights and a much more detailed comparison of the convergence using a broader variety of weights can be found in \cite{ImprOnePointQuadrData}. Moreover, it contains schemes which additionally choose suitable quadrature points to further improve convergence. This was ignored in this paper to keep the analysis simple. Accurate choices for partial volume weights for the case of $d=3$ can be found in \cite{Scabbia2023}. 
For the comparison, we calculated the norm $\|u_\kappa - u^*\|_{L^2(\Omega)^d}$ for the solutions $u_\kappa$ to the discrete problems and a reference solution $u^*$ calculated by using a relatively small $\kappa^*$.  The first problem is defined by 
\begin{equation}
\label{eqn:defprob1}
    k\equiv 100, l\equiv 800,
\end{equation}
as well as $\delta = \tfrac{1}{20}$, $\nu_{} \equiv 0$, $\Omega = [0,2]\times [0,1]$, $\Theta = [0,2\delta]\times [0,1]$ and $b\equiv (100, 0)^T\chi_{[2-2\delta, 2]\times [0,1]}$. This models a short bar being pulled on the right and held in place on the left. The $\kappa_n$ where chosen as $\kappa_n = (40 + 20n)^{-1}$ for $n\in\{0,...7\}$ and $u^*$ was obtained using $\kappa^*=360^{-1}$ and the PAAC weights. \newline
An additional comparison for a second problem featuring discontinuous input data is also shown in Figure \ref{fig:comp}. It has the same data except for 
\begin{equation}
\label{eqn:descrincl}
    k(x) = \begin{cases}
        0.01 & x \in B_{0.3}\Big((1, 0.5)^\top\Big) \\
        100 & \text{else}.
    \end{cases},
    l(x) = \begin{cases}
        0.08 & x \in B_{0.3}\Big((1,0.5)^\top\Big) \\
        800 & \text{else}.
    \end{cases}
\end{equation}

\begin{figure}[ht]
    \centering
    \includegraphics[width=\linewidth]{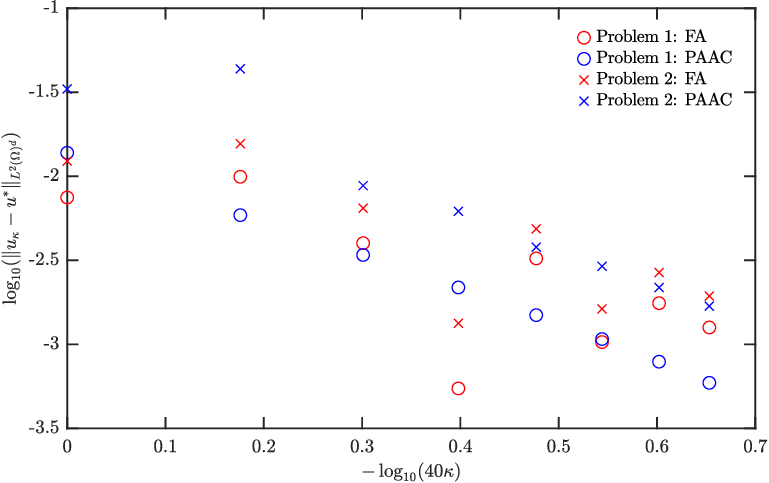}
    \caption{Comparison of the convergence behavior between the two choices of partial area weights as given in (\ref{eqn:defweightsPA}). A sequence of solutions for increasingly finer discretizations of a state-based model was calculated for both choices of the partial area weights (See (\ref{eqn:defweightsPA})) for two distinct problems defined by (\ref{eqn:defprob1}) and (\ref{eqn:descrincl}). The $y$-axis shows the logarithm of the $L^2(\Omega)^d$ norm of the difference between the solution $\b{u}_\kappa$ and a reference solution $u^*$ calculated at $\kappa = 360^{-1}$ using the PAAC weights. The $x$-axis was chosen to be a logarithmic scale for $\kappa$, shifted such that $\kappa_0=40^{-1}$ appears at $0$. The data of the first problem is shown plotted by the circles, while the crosses correspond to the second problem. The blue markers correspond to the solutions using the $w_{ij}^{\text{PAAC}}$ weights, the red markers correspond to the solutions for the $w_{ij}^{\text{FA}}$ weights. }
    \label{fig:comp}
\end{figure}
As visible in Figure \ref{fig:comp}, the models with the weights, namely the $w^{\text{PAAC}}_ij$ weights, demonstrate a more consistent convergence behavior. For a more detailed description of the oscillatory behavior found in the plots for the $\omega^{FA}$ weights, we refer to \cite{ImprOnePointQuadrData}. 
\end{remark}

%% file: note.tex
This research was funded by the Deutsche Forschungsgemeinschaft (DFG, German Research Foundation) - 377472739/GRK 2423/2-2023. The authors are very grateful for this support. 

The authors declare that they have no conflict of interest.

%% file: cnt/appendix.tex
\begin{lemma}\label{lemm:prfcube} Lemma \ref{lemm:zerosetcubeslemma} holds true.\end{lemma}

\begin{proof} We use \cite{MeasuTheo}, §3, Satz 8, with a slight modification, namely replacing the closed cubes with open ones by suitably extending them then taking their interior. We then obtain a series of open cubes $W_n$ with rational side lengths $p_n/q_n,\,$ $p_n, q_n \in \mathbb{N}$ for every $\epsilon >0$ such that 
\begin{equation*}
    A \subset \bigcup_{n\in \mathbb{N}} W_n, \, \vol\left(\bigcup_{n\in \mathbb{N}} W_n\right) < \epsilon.
\end{equation*}
Due to the compactness of $A$ and the openness of the $W_n$, only a finite subset of cubes is sufficient. Let $F_\epsilon$ be the set of indices of this finite subset of cubes. Then, subdivide each cube into smaller cubes of side length $(\text{lcm}_{n\in F_\epsilon}q_n)^{-1}$. Extending these cubes to half-closed cubes of equal volume then finishes the proof. 
\end{proof}

\begin{lemma}\label{lemm:SplitSBOp} The statement in Lemma \ref{lemm:OpSplitSb} holds. \end{lemma}

\begin{proof} Let $x\in \Omega_i$ and $\vec{u} \in {\mathbb{R}^d}^{|I^\kappa|}$, then
\begin{equation*}
    \begin{aligned}
        -[\ca{L}^\num_\kappa \Kk \vec{u}](x)&=\frac{1}{V_i} \sum_{j\not= i} B_{ij}^\num \vec{u}_j + \frac{1}{V_i} B_{ii}^\num \vec{u}_i\\
        &= \frac{1}{V_i}\sum_{j\not=i}\left(\sum_{m\not\in\{i,j\}}\ca{I}^{2,\num}_{mij} - \sum_{m\not=i} \ca{I}^{2,\num}_{imj}-\sum_{m\not=j}\ca{I}^{2,\num}_{jim}- \ca{I}^{1,\num}_{ij}\right)\vec{u}_j \\
        &\quad+  \frac{1}{V_i}\left(\sum_{n\not=i}\sum_{m\not=i}  \ca{I}^{2,\num}_{inm} + \sum_{m\not=i}\ca{I}^{2,\num}_{mii} + \sum_{m\not=i} \ca{I}^{1,\num}_{im}\right)\vec{u}_i.
    \end{aligned}
\end{equation*}
Separating out the terms involving $\ca{I}^{1,\num}$, one has
\begin{equation*}
    \begin{aligned}
        &\frac{1}{V_i}\sum_{j\not=i} \ca{I}^{1,\num}_{ij} (\vec{u}_i-\vec{u}_j) \\
        &=\frac{1}{V_i}\sum_{j\not=i}  V_iV_j^{\supscr{i}}\rho(x_j-x_i)(\alpha^\num(x_i)+\alpha^\num(x_j))\frac{\left(x_j-x_i\right)\left(x_j-x_i\right)^\top}{\|x_j-x_i\|\|x_j-x_i\|}(\vec{u}_i-\vec{u}_j) \dd x' \dd x,
    \end{aligned}
\end{equation*}
which is the $-\ca{L}^\num_1$ term. For the rest of the terms,
\begin{align*}
    &\frac{1}{V_i}\sum_{j\not=i}\left(\sum_{m\not\in\{i,j\}}\ca{I}^{2,\num}_{mij} - \sum_{m\not=i} \ca{I}^{2,\num}_{imj}-\sum_{m\not=j}\ca{I}^{2,\num}_{jim}\right)\vec{u}_j \\
    &\quad+\frac{1}{V_i}\left(\sum_{n\not=i}\sum_{m\not=i}  \ca{I}^{2,\num}_{inm} + \sum_{m\not=i}\ca{I}^{2,\num}_{mii}\right)\vec{u}_i\\
    &=\frac{1}{V_i}\sum_{j\not=i}\Bigg(\sum_{m\not\in\{i,j\}}\callInteg{m}{i}{j} \\
    &\phantom{======}-\sum_{m\not=i} \callInteg{i}{m}{j} \\
    &\phantom{======}-\sum_{m\not=j}\callInteg{j}{i}{m}\Bigg)\vec{u}_j \\
    &\quad+  \frac{1}{V_i}\Bigg(\sum_{n\not=i}\sum_{m\not=i}  \callInteg{i}{n}{m} \\
    &\phantom{======\,}+ \sum_{m\not=i}\callInteg{m}{i}{i}\Bigg)\vec{u}_i \\
    &=- \ca{L}^{\num,x}_2 (\vec{u}_i-\vec{u}_j) \\
    &\phantom{=}+\frac{1}{V_i}\sum_{j\not=i}\Bigg(\sum_{m\not\in\{i,j\}}\callInteg{m}{i}{j} \\
    &\phantom{======\,}-\sum_{m\not=\{j,i\}}\callInteg{j}{i}{m}\Bigg)\vec{u}_j \\
    &\quad+ \frac{1}{V_i}\Bigg(\sum_{j\not=i}\callInteg{j}{i}{i}\Bigg)(\vec{u}_i-\vec{u}_j) \\
    &= -\ca{L}^{\num, x}_2 (\vec{u}_i-\vec{u}_j)-\ca{L}^{\num, x'}_2 (\vec{u}_i-\vec{u}_j)-\ca{L}^{\num, p}_2 (\vec{u}_i-\vec{u}_j) \\
    &\quad+\frac{1}{V_i}\Bigg(\sum_{j\not=i}\sum_{m\not\in\{i,j\}}\callInteg{m}{i}{j} \\
    &\phantom{======\,}-\sum_{j\not=i}\sum_{m\not=\{j,i\}}\callInteg{j}{i}{m}\Bigg)\vec{u}_i,
\end{align*}
where, by swapping indices in the sums on the last two lines,
\begin{align*}
    &\frac{1}{V_i}\Bigg(\sum_{j\not=i}\sum_{m\not\in\{i,j\}}\callInteg{m}{i}{j} \\
    &\quad-\sum_{j\not=i}\sum_{m\not=\{j,i\}}\callInteg{j}{i}{m}\Bigg)\\
    &=\frac{1}{V_i}\Bigg(\sum_{j\not=i}\sum_{m\not\in\{i,j\}}\callInteg{m}{i}{j} \\
    &\quad-\sum_{j\not=i}\sum_{m\not=\{j,i\}}\callInteg{m}{i}{j}\Bigg)=0.
\end{align*}\end{proof}

\begin{lemma}\label{lemm:potform} The claims in Lemma \ref{lemm:PotFormLemma} hold true.\end{lemma}

\begin{proof}
We begin by using the explicit forms shown in Lemma \ref{lemm:SplitSBOp},
\begin{align*}
    &\phantom{=\,} \vec{u}^\top \vec{B}^\num \vec{u} = \langle \Kk \vec{u}, -\ca{L}^\num \Kk \vec{u}\rangle_{L^2(\Omega)^d} \\
    &=\sum_{i} \int_{\Omega_i}\vec{u}_i^\top\sum_{j\not=i} \int_{\Omega_j}w_{ij} \rho(x_j-x_i)(\alpha^\num(x_i)+\alpha^\num(x_j))\tfrac{\left(x_j-x_i\right)\left(x_j-x_i\right)^\top}{\|x_j-x_i\|\|x_j-x_i\|}(\vec{u}_i-\vec{u}_j) \dd x' \dd x \\
    &\phantom{=}+\sum_{i}\int_{\Omega_i} \vec{u}_i^\top\sum_{j\not=i}\int_{\Omega_j}\sum_{m\not\in\{i,j\}}\int_{\Omega_m}w_{mj}w_{mi}\tau^\num(x_{m})\\
    &\phantom{===========}\times \frclauseSB{i}{m}{j} (\vec{u}_j-\vec{u}_i) \dd p  \dd x' \dd x\\
    &\phantom{=}-\sum_{i} \int_{\Omega_i} \vec{u}_i^\top\sum_{j\not=i}\int_{\Omega_j}\sum_{m\not=i} \int_{\Omega_m}w_{ij}w_{im}\tau^\num(x_{i})\\
    &\phantom{===========}\times  \frclauseSB{m}{i}{j} (\vec{u}_j-\vec{u}_i) \dd p  \dd x' \dd x \\
    &\phantom{=}-\sum_{i} \int_{\Omega_i}\vec{u}_i^\top\sum_{j\not=i}\int_{\Omega_j}\sum_{m\not=j}\int_{\Omega_m}w_{jm}w_{ji}\tau^\num(x_{j})\\
    &\phantom{===========}\times \frclauseSB{i}{j}{m}(\vec{u}_j-\vec{u}_i) \dd p  \dd x' \dd x.
\end{align*}
The first summand then easily transforms into the following by splitting the sum at the $\alpha$ terms and renaming the constants,
\begin{align*}
    &\sum_{i} \int_{\Omega_i}\vec{u}_i^\top\sum_{j\not=i}\int_{\Omega_j}w_{ij}\rho(x_j-x_i)(\alpha^\num(x_i)+\alpha^\num(x_j))\frac{\left(x_j-x_i\right)\left(x_j-x_i\right)^\top}{\|x_j-x_i\|\|x_j-x_i\|}(\vec{u}_i-\vec{u}_j) \dd x' \dd x \\
    &=  \sum_{i} \int_{\Omega_i}\vec{u}_i^\top\sum_{j\not=i} \int_{\Omega_j}w_{ij}\rho(x_j-x_i)\alpha^\num(x_i)\frac{\left(x_j-x_i\right)\left(x_j-x_i\right)^\top}{\|x_j-x_i\|\|x_j-x_i\|}(\vec{u}_i-\vec{u}_j) \dd x' \dd x  \\
    &\quad+\sum_{i} \int_{\Omega_i}\vec{u}_i^\top\sum_{j\not=i} \int_{\Omega_j}w_{ij} \rho(x_j-x_i)\alpha^\num(x_j)\frac{\left(x_j-x_i\right)\left(x_j-x_i\right)^\top}{\|x_j-x_i\|\|x_j-x_i\|}(\vec{u}_i-\vec{u}_j) \dd x' \dd x\\
    &= \sum_{i} \int_{\Omega_i}(\vec{u}_i-\vec{u}_j)^\top\sum_{j\not=i} \int_{\Omega_j}w_{ij}\rho(x_j-x_i)\alpha^\num(x_i)\frac{\left(x_j-x_i\right)\left(x_j-x_i\right)^\top}{\|x_j-x_i\|\|x_j-x_i\|}(\vec{u}_i-\vec{u}_j) \dd x' \dd x \\
    &= \sum_{i} \int_{\Omega_i}\alpha^\num(x_i)\sum_{j\not=i} \int_{\Omega_j}w_{ij}\rho(x_j-x_i)\left(\frac{\left(x_j-x_i\right)^\top(\vec{u}_i-\vec{u}_j)}{\|x_j-x_i\|}\right)^2 \dd x' \dd x .
\end{align*}
The rest of the terms, after renaming and shifting the indices, are equal to
\begin{align*}
        &\sum_{j}\sum_{m\not=j}\sum_{i\not\in\{m,j\}}\int_{\Omega_j}\int_{\Omega_m}\int_{\Omega_i} \vec{u}_j^\top w_{im}w_{ij}\tau^\num(x_{i})\\
        &\phantom{===========}\times \frclauseSB{j}{i}{m} (\vec{u}_m-\vec{u}_j) \dd p  \dd x' \dd x\\
        &\quad-\sum_{i}\sum_{j\not=i}\sum_{m\not=i}\int_{\Omega_i}\int_{\Omega_j}\int_{\Omega_m} \vec{u}_i^\top w_{im}w_{ij}\tau^\num(x_{i})\\
        &\phantom{===========}\times \frclauseSB{m}{i}{j}(\vec{u}_j-\vec{u}_i) \dd p  \dd x' \dd x \\
        &\quad+\sum_{j}\sum_{i\not=j}\sum_{m\not=i}\int_{\Omega_j}\int_{\Omega_i}\int_{\Omega_m} \vec{u}_j^\top w_{im}w_{ij}\tau^\num(x_{i})\\
        &\phantom{===========}\times\frclauseSB{j}{i}{m}(\vec{u}_j-\vec{u}_i) \dd p  \dd x' \dd x \\
        \intertext{adding summand 1 and 3 together after rearranging the sums yields}
        &{=}\sum_{j}\sum_{m\not=j}\sum_{i\not\in\{m,j\}}\int_{\Omega_j}\int_{\Omega_m}\int_{\Omega_i} \vec{u}_j^\top w_{im}w_{ij}\tau^\num(x_{i})\\
        &\phantom{===========}\times \frclauseSB{j}{i}{m} (\vec{u}_m-\vec{u}_i) \dd p  \dd x' \dd x\\
        &\quad-\sum_{i}\sum_{j\not=i}\sum_{m\not=i}\int_{\Omega_i}\int_{\Omega_j}\int_{\Omega_m} \vec{u}_i^\top w_{im}w_{ij}\tau^\num(x_{i})\\
        &\phantom{===========}\times \frclauseSB{m}{i}{j} (\vec{u}_j-\vec{u}_i) \dd p  \dd x' \dd x \\
        &\quad+\sum_{j}\sum_{i\not=j}\int_{\Omega_j}\int_{\Omega_i}\int_{\Omega_j} \vec{u}_j^\top w_{ij}^2\tau^\num(x_{i})\\
        &\phantom{===========}\times\frclauseSB{j}{i}{j}(\vec{u}_j-\vec{u}_i) \dd p  \dd x' \dd x 
        \intertext{adding the first two summands after rearranging the sums gives }
        &{=}\sum_{j}\sum_{m\not=j}\sum_{i\not\in\{j,m\}}\int_{\Omega_j}\int_{\Omega_m}\int_{\Omega_i}w_{ij}w_{im}\tau^\num(x_{i})\rho(x_j-x_i)\left(x_{j}-x_{i}\right)^\top(\vec{u}_j-\vec{u}_i)\\
        &\phantom{==================}\times\rho(x_m-x_i)\left(x_{m}-x_{i}\right)^\top(\vec{u}_m-\vec{u}_i) \dd p  \dd x' \dd x \\
        &-\sum_{i}\sum_{j\not=i}\int_{\Omega_i}\int_{\Omega_j}\int_{\Omega_j} w_{ij}^2\vec{u}_i^\top\tau^\num(x_{i})\\
        &\phantom{===========}\times  \frclauseSB{j}{i}{j} (\vec{u}_j-\vec{u}_i) \dd p  \dd x' \dd x \\
        &+\sum_{j}\sum_{i\not=j}\int_{\Omega_j}\int_{\Omega_i}\int_{\Omega_j} w_{ij}^2\vec{u}_j^\top\tau^\num(x_{i})\\
        &\phantom{===========}\times  \frclauseSB{j}{i}{j}  (\vec{u}_j-\vec{u}_i) \dd p  \dd x' \dd x
        \intertext{adding all summands together, first the last two, then the rest, leads to}
        &\stackrel{}{=}\sum_{i}\int_{\Omega_i}\tau^\num(x_{i})\sum_{j\not= i}\int_{\Omega_j}w_{ij}\rho(x_{j}-x_{i})\left(x_{j}-x_{i}\right)^\top(\vec{u}_j-\vec{u}_i)  \dd x'\\
        &\phantom{===========}\times \sum_{m\not=i}\int_{\Omega_m} w_{im}\rho(x_m-x_i)\left(x_{m}-x_{i}\right)^\top(\vec{u}_m-\vec{u}_i) \dd p \dd x\\
        &\stackrel{}{=}\sum_{i}\int_{\Omega_i}\tau^\num(x_{i})\left(\sum_{j\not= i}\int_{\Omega_j}w_{ij}\rho(x_{j}-x_{i})\left(x_{j}-x_{i}\right)^\top(\vec{u}_j-\vec{u}_i)  \dd x'\right)^2 \dd x.
\end{align*}
Together with the previous terms, we obtain the following form of $\vec{u}^\top \vec{B}^\num \vec{u}$,
\begin{align*}
        &\sum_{i}\int_{\Omega_i}\left(k(x_i) -\frac{l(x_i)}{d^2}\right)\left(\frac{d}{m^\num(x_i)}\sum_{j\not= i}\int_{\Omega_j}w_{ij}\rho(x_{j}-x_{i})\left(x_{j}-x_{i}\right)^\top(\vec{u}_j-\vec{u}_i)  \dd x'\right)^2 \\
        &\quad+ \alpha^\num(x_i)\sum_{j\not=i} \int_{\Omega_j}w_{ij}\rho(x_{j}-x_{i})\left(\frac{\left(x_j-x_i\right)^\top(\vec{u}_j-\vec{u}_i)}{\|x_j-x_i\|}\right)^2 \dd x' \dd x\\
        &=\sum_{i}\int_{\Omega_i}k(x_i)\left(\frac{d}{m^\num(x_i)}\sum_{j\not= i}\int_{\Omega_j}w_{ij}\rho(x_{j}-x_{i})\left(x_{j}-x_{i}\right)^\top(\vec{u}_j-\vec{u}_i)  \dd x'\right)^2 \\
        &\quad -2\frac{\alpha^\num(x_i)m^\num(x_i)}{d^2}\left(\frac{d}{m^\num(x_i)}\sum_{j\not= i}\int_{\Omega_j}w_{ij}\rho(x_{j}-x_{i})\left(x_{j}-x_{i}\right)^\top(\vec{u}_j-\vec{u}_i)  \dd x'\right)^2 \\
        &\quad+\alpha^\num(x_i)m^\num(x_i)\left(\frac{1}{m^\num(x_i)}\sum_{j\not= i}\int_{\Omega_j}w_{ij}\rho(x_{j}-x_{i})\left(x_{j}-x_{i}\right)^\top(\vec{u}_j-\vec{u}_i)  \dd x'\right)^2 \\
        &\quad+ \alpha^\num(x_i)\sum_{j\not=i} \int_{\Omega_j}w_{ij}\rho(x_j-x_i)\left(\frac{\left(x_j-x_i\right)^\top(\vec{u}_j-\vec{u}_i)}{\|x_j-x_i\|}\right)^2 \dd x' \dd x  \\         
        &=\sum_{i}\int_{\Omega_i}k(x_i)\left(\frac{d}{m^\num(x_i)}\sum_{j\not= i}\int_{\Omega_j}w_{ij}\rho(x_{j}-x_{i})\left(x_{j}-x_{i}\right)^\top(\vec{u}_j-\vec{u}_i)  \dd x'\right)^2 \\
        &\quad -2\alpha^\num(x_i)\left(\sum_{j\not= i}\int_{\Omega_j}w_{ij}\rho(x_{j}-x_{i})\left(x_{j}-x_{i}\right)^\top(\vec{u}_j-\vec{u}_i)  \dd x'\right) \\
        &\phantom{=========}\times\left(\frac{1}{m^\num(x_i)}\sum_{j\not= i}\int_{\Omega_j}w_{ij}\rho(x_{j}-x_{i})\left(x_{j}-x_{i}\right)^\top(\vec{u}_j-\vec{u}_i)  \dd x'\right)\\       
        &\quad+\alpha^\num(x_i)\sum_{j\not= i}\int_{\Omega_j}w_{ij}\rho(x_j-x_i)\|x_j-x_i\|^2\\
        &\phantom{=========}\left(\frac{1}{m^\num(x_i)}\sum_{m\not= i}\int_{\Omega_j}w_{im}\rho(x_{m}-x_{i})\left(x_{m}-x_{i}\right)^\top(\vec{u}_m-\vec{u}_i)  \dd p\right)^2 \dd x' \\
        &\quad+ \alpha^\num(x_i)\sum_{j\not=i} \int_{\Omega_j}w_{ij}\rho(x_j-x_i)\|x_j-x_i\|^2\left(\frac{\left(x_j-x_i\right)^\top(\vec{u}_j-\vec{u}_i)}{\|x_j-x_i\|^2}\right)^2 \dd x' \dd x  \\
        &=\sum_{i}\int_{\Omega_i}k(x_i)\left(\frac{d}{m^\num(x_i)}\sum_{j\not= i}\int_{\Omega_j}w_{ij}\rho(x_{j}-x_{i})\left(x_{j}-x_{i}\right)^\top(\vec{u}_j-\vec{u}_i)  \dd x'\right)^2 \\
        &\quad+ \alpha^\num(x_i)\sum_{j\not=i} \int_{\Omega_j}w_{ij}\rho(x_j-x_i)\|x_j-x_i\|^2\Bigg(\frac{\left(x_j-x_i\right)^\top(\vec{u}_j-\vec{u}_i)}{\|x_j-x_i\|^2} \\
        &\phantom{=========}-\frac{1}{m^\num(x_i)}\sum_{m\not= i}\int_{\Omega_m}w_{im}\rho(x_{m}-x_{i})\left(x_{m}-x_{i}\right)^\top(\vec{u}_m-\vec{u}_i)  \dd p\Bigg)^2 \dd x'
\end{align*}

\end{proof}